\documentclass[10pt]{amsart}
\usepackage{mathrsfs}
\usepackage{amssymb}
\usepackage{amsfonts}
\usepackage{amsmath}
\usepackage{amsthm}
\usepackage{cases}
\usepackage{bbding}
\usepackage{stmaryrd}
\usepackage{graphicx}
\usepackage{verbatim}
\usepackage{enumerate} 
\usepackage{empheq} 
\usepackage{bbm}
\usepackage{comment}
\usepackage[colorlinks, linkcolor=red, citecolor=blue]{hyperref}
\usepackage[pagewise]{lineno}
\allowdisplaybreaks

\title[Invariant tori for area-preserving maps]
{Invariant tori for area-preserving maps with ultra-differentiable perturbation and Liouvillean frequency}
\author{Hongyu Cheng}
\address{
School of Mathematical Sciences, Tiangong University \\
Tianjin, 300071, P.R. China} \email{hychengmath@tiangong.edu.cn}
\author{Shimin Wang}
\address{
School of Mathematics, Shandong University\\
Jinan 250100, China} \email{shmwang@sdu.edu.cn}
\author{Fenfen Wang}
\address{
School of Mathematical Sciences and Laurent Mathematics Center, Sichuan Normal University, Chengdu 610066, P.R.China.} \email{ffenwang@hotmail.com}

\usepackage{color}

\newtheorem{theorem}{Theorem}[section]
\newtheorem{lemma}{Lemma}[section]
\newtheorem{definition}{Definition}[section]
\newtheorem{proposition}[lemma]{Proposition}

\numberwithin{equation}{section}
\theoremstyle{plain}

\DeclareMathOperator{\meas}{meas}

\newcommand{\NN}{\mathbb{N}}
\newcommand{\RR}{\mathbb{R}}

\newcommand{\TT}{\mathbb{T}}

\newcommand{\me}{\mathrm{e}}   
\newcommand{\mi}{\mathrm{i}}     
\newcommand{\dif}{\mathrm{d}}

\begin{document}
\begin{abstract}
We prove the existence of invariant tori to the area-preserving maps defined on $ \mathbb{R}^2\times\mathbb{T} $
	\begin{equation*}
		\overline{x}=F(x,\theta),
		\qquad \overline{\theta}=\theta+\alpha\, \,(\alpha\in \mathbb{R}\setminus\mathbb{Q}),
	\end{equation*}
where $ F $ is closed to a linear rotation, and the perturbation
is ultra-differentiable in $ \theta\in \mathbb{T},$ which is very closed to $C^{\infty}$ regularity.
Moreover, we assume that the frequency $\alpha$ is any irrational number without other arithmetic conditions and the smallness of the perturbation does not depend on $\alpha$.
Thus, both the difficulties from the ultra-differentiability of the perturbation and Liouvillean frequency will appear in this work. The proof of the main result is based on the Kolmogorov-Arnold-Moser (KAM) scheme about the area-preserving maps with some new techniques.
\end{abstract}

.
\maketitle	


\section{Introduction and main result}\label{sec:intro}
For the area-preserving maps, the Moser's Twist-mapping Theorem (\cite{Moser62,Siegelm71}) guarantees the existence of invariant simply closed curves surrounding the elliptic fixed point with Diophantine frequency.
In \cite{Russmann02}, H. R\"{u}ssmann, by assuming that the frequency of linearization satisfies Brjuno condition and employing the Moser's Twist-mapping Theorem, proved the stability of an elliptic fixed point of the area-preserving maps.
The authors, in \cite{delaLlave15,delaLlave06,delaLlave066,delaLlave03,delaLlave033}, constructed the quasi-periodic solutions and invariant manifolds to the quasi-periodic maps around the hyperbolic fixed point by the parameterization method. Moreover, in \cite{KAvila06,AvilaJ09,AvilaF11,HouY12}, the authors proved the reducibility of quasi-periodic $SL(2, \mathbb{R})$ cocycles and quasi-periodic linear ODEs in $sl(2,\mathbb{R})$ with frequency
$\omega=(1,\alpha),\ \forall \alpha\in\mathbb{R}\setminus\mathbb{Q}.$
Recently, in \cite{KrikorianW18}, Krikorian-Wang-You-Zhou extended the result in \cite{AvilaF11} to the nonlinear system by studying the linearization of the circle flow. Later, Wang-You-Zhou in \cite{WangY17,WangY16}, constructed the response solutions to the nonlinear finite-dimensional harmonic oscillators and investigated the boundedness of nonlinear differential equation with Liouvillean frequency $\omega=(1,\alpha),\ \forall \alpha\in\mathbb{R}\setminus\mathbb{Q}.$

The regularities with $\theta$ in the results mentioned above are all analytic.
In \cite{Cheng22}, the authors recently proved the reducibility of quasi-periodic $SL(2, \mathbb{R})$ cocycles in the ultra-differentiable setting, i.e., the function for the fiber component just ultra-differentiable in the variable $\theta\in\mathbb{T}$.
Motivated by the above works, we consider the area-preserving map in the ultra-differentiable setting:
\begin{equation}\label{map}
	\overline{x}=F(x,\theta),
	\qquad \overline{\theta}=\theta+\alpha.
\end{equation}
Here the skew-product system \eqref{map} is a bundle map on $ \mathbb{R}^2 \times \mathbb{T} $ ($ \mathbb{T} = \mathbb{R} / \mathbb{Z} $), and $\alpha\in \mathbb{R}\setminus\mathbb{Q}$ is the rotation vector.
Note that the area-preserving map means $ F(\cdot,\theta) : \mathbb{R}^2 \rightarrow \mathbb{R}^2 $ is area-preserving for each $ \theta \in \mathbb{T} $.
So we call \eqref{map} as an area-preserving system.
We assume that eigenvalues of $ DF(0,\theta) $ are modulus of $ 1 $.
More concretely, the function $F$ depending on $ \lambda \in \mathcal{O} = [\frac{1}{4}, \frac{3}{4}] $ has the decomposition
\begin{equation}\label{decomposition}
	F(x, \theta) := F(x, \theta, \lambda) = L(\lambda) x + \varepsilon N(x, \theta, \lambda),
\end{equation}
where
\begin{equation*}
	L(\lambda) = \left(
		\begin{array}{cc}
			\cos 2\pi \lambda & -\sin 2\pi \lambda \\
			\sin 2\pi \lambda &  \cos 2\pi \lambda
		\end{array}
		\right).
\end{equation*}
Moreover, the perturbation $ N $ is just ultra-differentiable in $ \theta \in \mathbb{T} $ and $ C^{1}_{W}$ ($ C^{1} $-smooth in the sense of Whitney) in the parameter $ \lambda \in \mathcal{O} $.


We quantify the regularity of the function $f:\mathbb{T}^d\rightarrow \mathbb{R}$.
To this end, we introduce the weight function:	
\begin{equation*}
	\Lambda: [0,\infty)\rightarrow [0, \infty), \quad y\mapsto \Lambda(y),
\end{equation*}
which is continuous and strictly increasing on $\RR^{+}.$
Expanding a smooth function $f\in C^{\infty}(\mathbb{T}^d,\mathbb{R})$ in Fourier series:
\begin{equation*}
	f(\theta)=\sum_{k\in \mathbb{Z}^d} \widehat{f}_{k} \me^{\mathrm{i} 2\pi \langle k, \theta \rangle},
\end{equation*}
we say it is $\Lambda$-ultra-differentiable if there exists $r>0$, which we call a ``width'', such that
\begin{equation*}
	\|f\|_{r} := \sum_{k\in \mathbb{Z}^d} |\widehat{f}_{k}| \me^{\Lambda(2\pi |k|r)}<\infty.
\end{equation*}
Here $ |k| = \sum_{i=1}^{d} |k_i| $, $ \forall k \in \mathbb{Z}^d $.
The real-analytic and $\delta$-Gevrey regularity correspond to $\Lambda(y)=y$ and $\Lambda(y)=y^{\delta}\,(0<\delta<1)$, respectively.

An invariant torus for the skew-product system \eqref{map} is given by a section $K:\mathbb{T} \times \mathcal{O}\rightarrow \mathbb{R}^2 $ satisfying
\begin{equation}\label{embedding}
	F(K(\theta,\lambda),\theta,\lambda)=K(\theta+\alpha,\lambda).
\end{equation}
Our result shows that for the area-preserving map $F$ defined by \eqref{decomposition}, the system \eqref{map} has an invariant torus with frequency $\alpha$ when $\varepsilon$ is small enough.
 Our main result is stated in the following:
\begin{theorem}\label{mainthm}
	We consider the invariance equation \eqref{embedding} with $F$ given by \eqref{decomposition}.
	Let $U \subseteq \mathbb{R}^2$ be an open neighborhood around the origin.
	Assume that $N:U\times\mathbb{T}\times\mathcal{O}\rightarrow \mathbb{R}^2$ is analytic  in $ x \in  U $, $C_{W}^1$ in $\lambda\in \mathcal{O}$ and
	$\Lambda$-ultra-differentiable in $ \theta \in \mathbb{T}$.
	Moreover, assume that the weight function $ \Lambda $ satisfies
	\begin{itemize}
		\item [(\textbf{H1})]
			Subadditivity: $ \Lambda(x+y) \leq \Lambda(x) + \Lambda(y), \, \forall\,  x,y \in [0,\infty) $;
		\item [(\textbf{H2})]
			Logarithmic growth:  $ \Gamma(x):=\frac{x\Lambda'(x)}{\ln x} \rightarrow \infty $ monotonically as $x\rightarrow\infty.$
\end{itemize}
Then, for any given sufficiently small $\gamma>0$,  there exist $\varepsilon_{0}>0$ depending on $\gamma,\, N$ and a Cantor subset $\mathcal{O}_{\gamma}\subseteq\mathcal{O}$ with  ${\meas}(\mathcal{O}\setminus\mathcal{O}_{\gamma})=O(\gamma)$
	such that, for any $\lambda\in \mathcal{O}_{\gamma}$, there exists a $C^{\infty}$ map $ K(\cdot, \lambda):\mathbb{T} \rightarrow U $ satisfying the equation \eqref{embedding} provided
	$0<\varepsilon<\varepsilon_{0}$, i.e., the invariant torus for the area-preserving system \eqref{map}.
\end{theorem}

The assumptions (\textbf{H1}) and (\textbf{H2}) are not restrictive.
Indeed, the sub-additive condition (\textbf{H1}) is a very classical assumption in the literature, which guarantees the space of $\Lambda$-ultra-differentiable functions forms a Banach algebra.
In view of our choice of norms, the condition
(\textbf{H2}) is exactly what ensures we have an analogue of the Cauchy estimates for
analytic functions. Moreover, the space of $\Lambda$-ultra-differentiable functions with $\Lambda$ satisfying (\textbf{H1}) and (\textbf{H2}) is very close to the $C^{\infty}$ space. It covers the analytic class $\Lambda(x)=x,$ Gevrey class $\Lambda(x)=x^{\sigma} (0<\sigma<1),$ and other ultra-differentiable
classes $\Lambda(x)=\exp\{(\ln x)^{\sigma}\} (0<\sigma<1),$ $\Lambda(x)=(\ln x)^{\beta}(\beta>1).$

The fact that the map $F$ defined by \eqref{decomposition} is area-preserving  plays a fundamental role in our discussions. We will follow the philosophy dealing with cocycles in spectrum theory of linear operator to construct the area-preserving change of variables, see \cite{AvilaF11,Cheng22,Houwang20,Caiw21,Eliasson92,Zhangz17} and the references therein.
	
\section{Preliminaries}\label{sec:pre}
	
\subsection{Area-preserving map}
\label{subsectionone}
In this subsection, we will give a brief introduction about area-preserving map. The map, in the $(x,y)-$variables, is given by
\begin{equation}\label{modelmap}
	\begin{split}
		x_{1}&=f(x,y)=x\cos\gamma_{0}-y\sin\gamma_{0}+a_{11}x+a_{12}y
+\mathrm{h.o.t.},\\
y_{1}&=g(x,y)=x\sin\gamma_{0}+y\cos\gamma_{0}+a_{21}x+a_{22}y
+\mathrm{h.o.t.},
	\end{split}
\end{equation}
where $\mathrm{h.o.t.}$ are higher order terms in $x,y$ with real coefficients. The map \eqref{modelmap} is area-preserving if
\begin{equation}\label{area-preservingcondition}
	\det \frac{\partial(f,g)}{\partial(x,y)}(x,y)
=1, \quad  \forall(x,y)\in\RR^{2}.
\end{equation}
Particularly, set $(x,y)=(0,0)$ in \eqref{area-preservingcondition}, then we get
\begin{equation}\label{area-preservingconditions}
\det\Big\{\Big(
	\begin{array}{cc}
		\cos\gamma_{0}  & -\sin\gamma_{0} \\
		\sin\gamma_{0}  &  \cos\gamma_{0}
	\end{array}
	\Big)
+\Big(
	\begin{array}{cc}
		a_{11}  & a_{12} \\
		a_{21}  &  a_{22}
	\end{array}
	\Big)\Big\}=1.
\end{equation}
The equality in \eqref{area-preservingconditions} is one of fundamental ingredients in our discussions.

Define the area-preserving change of variables
\begin{equation*}
	\begin{split}
		x&=\phi(\xi,\eta)=\xi+\cdots,\quad x_{1}=\phi(\xi_{1},\eta_{1}),\\
		y&=\psi(\xi,\eta)=\eta+\cdots,\quad y_{1}=\psi(\xi_{1},\eta_{1}),\
		\phi_{\xi}\psi_{\eta}-\phi_{\eta}\psi_{\xi}=1,
	\end{split}
\end{equation*}
where $\phi(\xi,\eta)-\xi$ and $\psi(\xi,\eta)-\eta$
are formal power series in $\xi,\eta$ with real coefficients,
which begin with terms of second order and, under some conditions, transforms the map
\eqref{modelmap} into the Birkhoff normal form
\begin{equation*}
	\begin{split}
\xi_{1}&=\xi\cos\Gamma(\xi,\eta)-\eta\sin\Gamma(\xi,\eta),\\
		\eta_{1}&=\xi\sin\Gamma(\xi,\eta)+\eta\cos\Gamma(\xi,\eta),\ \ \
		\Gamma(\xi,\eta)=\sum_{j=0}^{\infty}\gamma_{j}(\xi^{2}+\eta^{2})^{j},
	\end{split}
\end{equation*}
where $\gamma_{j},j=1,\cdots,$ are the uniquely determined Birkhoff constants. See discussions in section $23$ of the book  \cite{Siegelm71} and \cite{Moser62,Moser67} for details.

\subsection{Functional setting}
Let $ \mathcal{O} = [\frac{1}{4}, \frac{3}{4}] $ and define the space
\begin{equation*}
	\mathcal{C}^2 = \big\{ V=(v,\overline{v})^{\mathrm{T}}:\,v\in \mathbb{C} \big\},
\end{equation*}
and
\begin{equation*}
	\mathcal{C}^{2\times 2} =
	\Big\{ W =\Big(
	\begin{array}{cc}
		W_1  & W_2 \\
		\overline{W}_{2}  &  \overline{W}_1
	\end{array}
	\Big): W_1,W_2\in \mathbb{C} \Big\},\
SU(1,1)=\{W\in\mathcal{C}^{2\times 2}:\mathrm{det}W=1\}.
\end{equation*}
The Lie algebra of $SU(1,1)$ is
\begin{equation*}
su(1,1)=\{W\in\mathcal{C}^{2\times 2}:W_1\in\mi\RR,W_2\in \mathbb{C}\}.
\end{equation*}

For a $C^{1}_{W}$
function $g:\mathcal{O}\rightarrow*$ ($*$ will usually
denote $\mathbb{R}$,\,$\mathbb{C}$, $\mathcal{C}^2$,\,$\mathcal{C}^{2\times 2}$), we define its norm as
\begin{equation*}
	|g|_{\mathcal{O}}:=\sup_{\lambda \in\mathcal{O}}\big(|g(\lambda)|+\big|\partial_{\lambda} g(\lambda)\big|\big).
\end{equation*}
For the function $f(\theta,\lambda)=\sum\limits_{k\in\mathbb{Z}}\widehat{f}_{k}(\lambda) \me^{\mathrm{i} 2\pi k \theta}$, which is $\Lambda$-ultra-differentiable in $\theta\in \mathbb{T}$ with the width $r$ and $C_{W}^1$ in $\lambda\in \mathcal{O}$, we denote
\begin{equation*}
	\Lambda_{r}( \mathbb{T}\times \mathcal{O},*) = \big\{
	f:\,\,\|f\|_{r,\mathcal {O}}=\sum_{k\in \mathbb{Z}}|\widehat{f}_{k}|_{\mathcal{O}} \me^{\Lambda(2\pi |k|r)}<\infty
	\big\},
\end{equation*}
and for $f(x,\theta,\lambda)=\sum\limits_{m\in\mathbb{N}^2}f_{m}(\theta,\lambda)x^m$, which is $\Lambda$-ultra-differentiable in $\theta\in \mathbb{T}$ with the width $r$, analytic in $x\in \{ x\in\mathcal{C}^2:\, \|x\|_{\mathcal{C}^2}<s \}$ and $C_{W}^1$ in $\lambda\in \mathcal{O}$, we denote
\begin{equation*}
	\Lambda_{r,s}(\mathcal{C}^2\times\mathbb{T}\times \mathcal{O},*)= \Big\{
	f: \,\,\|f\|_{r,s,\mathcal {O}}=\sup_{ \|x\|_{\mathcal{C}^2}< s}\sum_{m\in \mathbb{N}^2}\|f_m\|_{r,\mathcal {O}}\cdot |x|^{|m|}<\infty
	\Big\}.
\end{equation*}
Moreover, we define the truncation operator $\mathcal{T}_{K}$ and projection operator $\mathcal{R}_{K}$ as
\begin{equation*}
	\mathcal{T}_{K}f(\theta,\lambda):=\sum_{|k|< K}\widehat{f}_{k}(\lambda) \me^{\mathrm{i} 2\pi k\theta},
	\qquad
	\mathcal{R}_{K}f(\theta,\lambda):=\sum_{|k|\geq K}\widehat{f}_{k}(\lambda) \me^{\mathrm{i} 2\pi k\theta}.
\end{equation*}
The average $[f(\theta,\lambda)]_{\theta}$ of $f(\theta,\lambda)$ over $\mathbb{T}$ is defined as
\begin{equation*} [f(\theta,\lambda)]_{\theta}:=\int_{\mathbb{T}}f(\theta,\lambda) \dif \theta=\widehat{f}_0(\lambda).
\end{equation*}

\subsection{Continued fraction expansion}\label{sec:cdbridge}
Let us recall some properties of irrational number.
Given an irrational number $\alpha\in (0,1),$ we define
\begin{equation*}
	a_{0}=0,\quad \alpha_{0}=\alpha,
\end{equation*}
and inductively for $k \geq 1,$
\begin{equation*}
	a_{k}=[\alpha_{k-1}^{-1}],\quad \alpha_{k}=\alpha_{k-1}^{-1}-a_{k},
\end{equation*}
where $[\alpha]:=\max\{m\in \mathbb{Z}:\,m\leq \alpha\}$.

Let $p_{0}=0,p_{1}=1,q_{0}=1,q_{1}=a_{1},$ and recursively,
\begin{equation*}
	\begin{split}
		p_{k} = a_{k}p_{k-1}+p_{k-2},\ \
		q_{k} = a_{k}q_{k-1}+q_{k-2}.
	\end{split}
\end{equation*}
Then $ \{q_{n}\} $ is the sequence of denominators of the best rational approximations for $\alpha$. It satisfies
\begin{equation*}
	\|k\alpha\|_{\mathbb{T}}\geq\|q_{n-1}\alpha\|_{\mathbb{T}}, \,\,\text{ for } 1 \leq k<q_{n},
\end{equation*}
and
\begin{equation*}
	\frac{1}{q_{n}+q_{n+1}}<\|q_{n}\alpha\|_{\mathbb{T}}
\leq\frac{1}{q_{n+1}},
\end{equation*}
where $\|x\|_{\mathbb{T}}:=\inf\limits_{p\in\mathbb{Z}}|x-p|.$

For the sequence $\{q_{n}\}$, we fix a special subsequence $\{q_{n_{k}}\}$. For simplicity, we denote the subsequences
$\{q_{n_{k}}\}$ and $\{q_{n_{k}+1}\}$ by $\{Q_{k}\}$ and $\{\overline{Q}_{k}\}$, respectively. Next, we introduce the
concept of CD bridge which was first introduced in \cite{AvilaF11}.
\begin{definition}[CD bridge, \cite{AvilaF11}]\label{CDbridge}
	Let $0<\mathbb{A}\leq\mathbb{B}\leq\mathbb{C}$. We say that the pair of denominators
	$(q_{m},q_{n})$ forms a CD$ (\mathbb{A},\mathbb{B},\mathbb{C})$ bridge if
	\begin{itemize}
		\item $ q_{i+1} \leq q_{i}^{\mathbb{A}},\,\, \text{for}\,\, i=m,\cdots,n-1; $
		\item $ q_{m}^{\mathbb{C}} \geq q_{n} \geq q_{m}^{\mathbb{B}}.$
	\end{itemize}
\end{definition}
\begin{lemma}[Lemma $3.2$ in \cite{AvilaF11}]\label{bridgeestimate}
	For any $\mathbb{A}\geq1$, there exists a subsequence $\{Q_{k}\}$ such that $Q_{0}=1$
	and for each $k\geq0,$ $Q_{k+1}\leq\overline{Q}_{k}^{\mathbb{A}^{4}}$, either
	$\overline{Q}_{k}\geq Q_{k}^{\mathbb{A}}$, or the pairs $(\overline{Q}_{k-1},Q_{k})$ and
	$(Q_{k},Q_{k+1})$ are both CD$(\mathbb{A},\mathbb{A},\mathbb{A}^{3})$ bridge.
\end{lemma}
\begin{lemma}\label{bries}
	The sequence $ \{Q_{k}\} $ selected in Lemma~\ref{bridgeestimate} satisfies the estimates
\begin{equation}\label{202207120}
 Q_{k+1}\geq Q_{k}^{\mathbb{A}},\quad
 \overline{Q}_{k+1}\geq\overline{Q}_{k}^{\mathbb{A}}.
	\end{equation}
\end{lemma}
\begin{proof}
We consider the second inequality in \eqref{202207120} first. From Lemma~\ref{bridgeestimate}, we know that
	either $\overline{Q}_{k+1}\geq Q_{k+1}^{\mathbb{A}}\geq \overline{Q}_{k}^{\mathbb{A}}$, or
	$(\overline{Q}_{k},Q_{k+1})$ and
	$(Q_{k+1},Q_{k+2})$ are both CD$(\mathbb{A},\mathbb{A},\mathbb{A}^{3})$ bridges, which means $\overline{Q}_{k+1}\geq\overline{Q}_{k}^{\mathbb{A}}.$ With the same discussions above we also get the first inequality in \eqref{202207120}.
\end{proof}

\section{KAM theorem and proof of Theorem~\ref{mainthm}}\label{sec:kam}
In this section, we develop an abstract KAM theorem for the skew-product system  \eqref{map} with area-preserving map \eqref{decomposition}, i.e., to prove there is a solution to the equation \eqref{embedding}.	Then Theorem~\ref{mainthm} is an immediate result of the KAM theorem.

\subsection{KAM theorem}
Consider the invariance equation with the unknown function $X\in \Lambda_{r}(\mathbb{T}\times \mathcal{O},\mathcal{C}^2)$:
\begin{equation}\label{kam-eq} X(\theta+\alpha,\lambda)=A(\lambda)X(\theta,\lambda)+U(\theta,\lambda)+W(\theta,\lambda)X(\theta,\lambda)+R(X(\theta,\lambda),\theta, \lambda),
\end{equation}
where
\begin{equation}\label{202207193}
\begin{split}
U \in \Lambda_{r}&(\mathbb{T}\times \mathcal{O},\mathcal{C}^2),
\ W  \in \Lambda_{r}(\mathbb{T}\times \mathcal{O},\mathcal{C}^{2\times 2}),\ R\in \Lambda_{r,s}(\mathcal{C}^2\times\mathbb{T}\times \mathcal{O},\mathcal{C}^{2}),\\
A(\lambda)&=\mathrm{diag} \big\{ \me^{\mathrm{i} 2\pi \lambda}, \me^{- \mathrm{i} 2\pi \lambda} \big\}, \quad A(\lambda)+W(\theta,\lambda)\in SU(1,1),\\
&R(0,\theta,\lambda)= 0, \quad
			\partial_{1} R(0,\theta,\lambda)=0,\forall(\theta,\lambda)\in\TT\times\mathcal{O},
\end{split}
\end{equation}
such that the right side of equation \eqref{kam-eq} corresponds to an area-preserving system.
Our goal is to show that, for most of parameters $\lambda\in \mathcal{O}$ (in the Lebesgue measure sense), the equation \eqref{kam-eq} admits a solution with Liouvillean frequency $\alpha\in \mathbb{R}\setminus\mathbb{Q}$ provided that the functions $U,\, W,\, R$ are small enough. Then we state our KAM theorem.
\begin{theorem}\label{kamthm}
Let $\alpha\in \mathbb{R}\setminus\mathbb{Q}$ and $0<s,r<1.$
Consider the invariance equation \eqref{kam-eq} with \eqref{202207193} and $\Lambda$ satisfies \textbf{(H1)} and \textbf{(H2)}.
Then, for any sufficiently small $\gamma>0$, there exists $\varepsilon_{0}>0$ (depending on $ r,\, s,\, \gamma $) such that whenever
	\begin{equation}\label{202207218}
		\begin{split}
			&\|U\|_{r,\mathcal{O}} \leq \varepsilon_{0},
			\quad \|W\|_{r,\mathcal{O}}\leq \varepsilon_{0}^{\frac{1}{2}},\quad
\|\partial^2_{11}R\|_{r,s,\mathcal{O}}\leq 1,
		\end{split}
	\end{equation}
	there exist a non-empty subset of $\lambda\in \mathcal{O}_{\gamma} \subseteq \mathcal{O}$ with $ \meas (\mathcal{O}\setminus\mathcal{O}_{\gamma})=O(\gamma),$ $V_{*}\in C^{\infty}(\mathbb{T}\times \mathcal{O}_{\gamma},\mathcal{C}^{2\times2}),\,\,R_{*}\in C^{\infty}(\mathcal{C}^2\times\mathbb{T}\times \mathcal{O}_{\gamma},\mathcal{C}^{2})$ and a $C^{\infty}$ area-preserving transformation $\Phi_{\ast}: C^{\infty}(\mathbb{T}\times \mathcal{O}_{\gamma}, \mathcal{C}^{2}) \rightarrow \Lambda_{r}(\mathbb{T}\times \mathcal{O}_{\gamma}, \mathcal{C}^{2})$ of the form
\begin{equation*}
X = \me^{\mathfrak{D}_{*}} X_{*} + \Xi_{*},
\end{equation*}
where $ \mathfrak{D}_{*} : \mathbb{T} \times
	\mathcal{O}_{\gamma} \rightarrow su(1,1)$ and $\Xi_{*} : \mathbb{T} \times
	\mathcal{O}_{\gamma} \rightarrow \mathcal{C}^{2},$ such that $ \Phi_{*} $ transforms equation \eqref{kam-eq} into
\begin{equation*}
 X_{*}(\theta+\alpha,\lambda)=(A(\lambda)
+V_{*}(\theta,\lambda))X_{*}(\theta,\lambda)
+R_{*}(X_{*}(\theta,\lambda),\theta,\lambda).
\end{equation*}
\end{theorem}

\subsection{Proof of Theorem~\ref{mainthm}: Application of the KAM theorem}
In this section, we prove our main Theorem~\ref{mainthm} by applying the KAM theorem~\ref{kamthm}.

Considering the invariance equation \eqref{embedding} with the area-preserving map \eqref{decomposition}, by Taylor's expansion, we have
\begin{equation}\label{taylorexpansion}
	\begin{split}
		K(\theta+\alpha,\lambda) &=L(\lambda)K(\theta,\lambda)+\varepsilon N(K(\theta,\lambda),\theta,\lambda)\\ &=L(\lambda)K(\theta,\lambda)+V(\theta,\lambda)
+S(\theta,\lambda)K(\theta,\lambda)+ P(K(\theta,\lambda),\theta,\lambda),
	\end{split}
\end{equation}
where
\begin{equation}\label{detail}
	\begin{split}
		V & = (V_1,V_2)^{\mathrm{T}} = \varepsilon N(0,\theta,\lambda)\in \mathbb{R}^2, \\
		S & = \left(
		\begin{array}{cc}
			S_{11}  &  S_{12}\\
			S_{21} &  S_{22}
		\end{array}
		\right)
		= \varepsilon \partial_{1} N(0,\theta,\lambda)\in \mathbb{R}^{2\times 2},\\
	P & = (P_{1},P_{2})^{\mathrm{T}}
	= \varepsilon N(K(\theta,\lambda),\theta,\lambda)
	- V(\theta,\lambda)
	- S(\theta,\lambda) K(\theta,\lambda) \in \mathbb{R}^2.
	\end{split}
\end{equation}

Set $ M=\frac{1}{\sqrt{2}} \left(
\begin{array}{cc}
	1 & \mathrm{i} \\
	1 & -\mathrm{i}
\end{array}
\right),$ and make the change $K=M^{-1}X,$
 then the area-preserving map $F$ defined by \eqref{decomposition} is conjugated to
\begin{equation}\label{202207194}
	\begin{split}
 M F (M^{-1} &(v, \bar{v})^{\mathrm{T}}, \theta, \lambda)\\
		&:= A(\lambda)(v, \bar{v})^{\mathrm{T}} + U(\theta,\lambda) + W(\theta,\lambda) (v, \bar{v})^{\mathrm{T}} + R((v, \bar{v})^{\mathrm{T}}, \theta, \lambda),
	\end{split}
\end{equation}
where
\begin{equation}\label{transformeddetail}
	\begin{split}
		A(\lambda) &= ML(\lambda)M^{-1} =  \mathrm{diag} \big\{ \me^{\mathrm{i} 2\pi \lambda}, \me^{-\mathrm{i} 2\pi \lambda} \big\} \in SU(1,1),\\
		U &= MV =\frac{1}{\sqrt{2}}(V_1+\mathrm{i} V_2,V_1-\mathrm{i} V_2)^{\mathrm{T}}\in \mathcal{C}^2,\\
		R &= M P =\frac{1}{\sqrt{2}}(P_1+\mathrm{i} P_{2},P_{1}-\mathrm{i} P_{2})^{\mathrm{T}} \in \mathcal{C}^2,
	\end{split}
\end{equation}
and
\begin{equation*}
	\begin{split}
		W(\theta,\lambda) = MS(\theta,\lambda)M^{-1} = \frac{1}{2}\left(
			\begin{array}{cc}
				W_1 & W_2 \\
				\overline{W}_{2} & \overline{W}_1
			\end{array}
			\right)\in\mathcal{C}^{2 \times 2},
	\end{split}
\end{equation*}
where $W_1=(S_{11}+S_{22})+\mathrm{i}(S_{21}-S_{12}),
	W_2=(S_{11}-S_{22})+\mathrm{i}(S_{21}+S_{12}).$
The map $ M F (M^{-1} (v, \bar{v})^{\mathrm{T}}, \theta, \lambda)$ defined by \eqref{202207194} is area-preserving since $M$ is a constant matrix. Correspondingly, \eqref{embedding} is changed into
\begin{equation}\label{taylorexpansion1}
	X(\theta+\alpha,\lambda)  = A(\lambda)X(\theta,\lambda) + U(\theta,\lambda) + W(\theta,\lambda) X(\theta,\lambda) + R(X(\theta,\lambda),\theta,\lambda).
\end{equation}

It follows from \eqref{detail} and \eqref{transformeddetail} that
\begin{equation*}
	\begin{split}
		&\qquad\qquad\|U\|_{r,\mathcal{O}} \leq c \varepsilon,\quad \,\, \|W\|_{r,\mathcal{O}}\leq  c \varepsilon\leq  \varepsilon^{\frac{1}{2}},\\
		& R(0,\theta,\lambda)= 0,
		\quad \partial_{1} R(0,\theta,\lambda)=0,\quad \|\partial^2_{11}R\|_{r,s,\mathcal{O}}\leq c \varepsilon\leq 1,
	\end{split}
\end{equation*}
where $ c $ is a constant depending on $ N $.
Then, by Theorem~\ref{kamthm}, for any sufficiently small $\gamma>0$,  there exists $\varepsilon_{0}>0$ (depending on $ s,\, r,\, \gamma $) and a positive Lebesgue
measure set of $ \lambda\in \mathcal{O}_{\gamma} \subseteq \mathcal{O}$ such that if $0<\varepsilon< \varepsilon_{0}$, there exists a $C^{\infty}$ map $\Phi_{*}$ transforming
the equation \eqref{taylorexpansion1} into
\begin{equation*}
	X_{*}(\theta+\alpha,\lambda)= ( A(\lambda)+V_{*}(\theta,\lambda) ) X_{*}(\theta,\lambda)+R_{*}(X_{*}(\theta,\lambda),\theta,\lambda),
\end{equation*}
which admits a trivial solution $X_{*}=0$ for each $ \lambda \in \mathcal{O}_{\gamma} $.
Then $X=(X_1,\overline{X}_{1})^{\mathrm{T}}=\Phi_{*}(X_*)\in \mathcal{C}^2$ is a solution of equation \eqref{taylorexpansion1}. Consequently,
\[ K(\theta,\lambda) = M^{-1}X(\theta,\lambda) = \sqrt{2}(\mathrm{Re} X_1(\theta,\lambda),\mathrm{Im} X_1(\theta,\lambda))^{\mathrm{T}}\in \mathbb{R}^2
\]
is a solution of the original equation \eqref{embedding} for each $ \lambda \in \mathcal{O}_{\gamma},$ i.e., the invariant torus with Liouvillean frequency $ \alpha $ for the area-preserving systems \eqref{map}.

\section{Solve the homological equation}\label{sec:1proofkam}
In this section we give precise procedures for deriving and solving the homological equations, which will be used in Section~\ref{sec:proofkam} to give one cycle of KAM iteration.

\subsection{Idea of deriving homological equations}\label{sec:outline}
Recall the equation \eqref{kam-eq}:
\begin{equation}\label{kam-eq1}
	X(\theta+\alpha, \lambda) = A(\lambda) X(\theta, \lambda) + U(\theta, \lambda) + W(\theta, \lambda) X(\theta, \lambda) + R(X(\theta, \lambda), \theta, \lambda),
\end{equation}
corresponding to an area-preserving system, here and after here we always set $A(\lambda)=\mathrm{diag} \big\{ \me^{\mathrm{i} 2\pi \lambda},\me^{-\mathrm{i} 2\pi \lambda}\big\}.$ Our goal is to look for an area-preserving transformation $\Phi $ of the form $ X = \me^{D} X_{+} + \Delta $ with $ D : \mathbb{T} \times \mathcal{O} \rightarrow su(1,1),$ $ \Delta : \mathbb{T} \times \mathcal{O} \rightarrow \mathcal{C}^{2},$ which transforms \eqref{kam-eq1} into
\begin{equation}\label{kam-eq2}
	\begin{split}
		X_{+}(\theta + \alpha, \lambda)
		&= \me^{-D(\theta + \alpha, \lambda)} A(\lambda) \me^{D(\theta, \lambda)} X_{+}(\theta, \lambda) \\
		& \quad + \me^{-D(\theta + \alpha, \lambda)} W(\theta, \lambda) \me^{D(\theta, \lambda)} X_{+}(\theta, \lambda)\\
		& \quad + \me^{-D(\theta + \alpha, \lambda)} \big[
		A(\lambda) \Delta(\theta, \lambda) + U(\theta, \lambda) - \Delta(\theta + \alpha, \lambda) \big]\\
		& \quad + \me^{-D(\theta + \alpha, \lambda)} W(\theta, \lambda) \Delta(\theta, \lambda) \\
		& \quad + \me^{-D(\theta + \alpha, \lambda)} R(\me^{D(\theta, \lambda)} X_{+}(\theta, \lambda) + \Delta(\theta, \lambda), \theta, \lambda),
	\end{split}
\end{equation}
where $U_{+},\,W_{+}$ are much smaller than $U,\,W.$

The key point in the KAM iteration is to find $ \Delta $ and $ D $ to solve the homological equations:
\begin{equation*}
	A(\lambda) \Delta(\theta,\lambda)-\Delta(\theta+\alpha,\lambda)=-U(\theta,\lambda),
\end{equation*}
and
\begin{equation}\label{homo-eq2}
	\me^{-D(\theta + \alpha, \lambda)} A(\lambda) \me^{D(\theta, \lambda)} + \me^{-D(\theta + \alpha, \lambda)} W(\theta, \lambda) \me^{D(\theta, \lambda)} = A_{+}(\lambda),
\end{equation}
where the unknown matrix $ A_{+} $ depends only on the parameter $ \lambda $.
Note $ \me^{M} = I + M + O(|M|^2) $ for any $ n \times n $ matrix $ M $, the homological equation~\eqref{homo-eq2} is
\begin{equation*}
	A(\lambda) D(\theta,\lambda) - D(\theta + \alpha, \lambda) A(\lambda) = -W(\theta, \lambda) + A_{+}(\lambda) - A(\lambda) + O(|D|^2|A|) + O(|D||W|).
\end{equation*}
Hence, if the equation
\begin{equation}\label{homo-eq22}
	A(\lambda) D(\theta,\lambda) - D(\theta + \alpha, \lambda) A(\lambda) = -W(\theta, \lambda) + A_{+}(\lambda) - A(\lambda)
\end{equation}
is solved,
then the equation~\eqref{homo-eq2} has an approximate solution with error
\begin{equation*}
	O(|D|^2|A|) + O(|D||W|),
\end{equation*}
which are terms of higher order in $ D $ and $W.$
Therefore, we actually solve the homological equations~\eqref{homo-eq22} instead of \eqref{homo-eq2}.

Set
$ D= \left( \begin{array}{cc}
		\mi D_1 & D_2  \\
		\overline{D}_2  & -\mi D_1
	\end{array} \right), W  = \left( \begin{array}{cc}
		W_1  & W_2 \\
		\overline{W}_2  &  \overline{W}_1
	\end{array} \right)$.
Then \eqref{homo-eq22} is changed into
\begin{equation*}
	\begin{split}
\me^{\mi 2\pi \lambda} D_{1}(\theta, \lambda) - \me^{\mi 2\pi \lambda} D_{1}(\theta + \alpha, \lambda) &= \mi W_{1}(\theta, \lambda), \\
\me^{\mi 2\pi \lambda} D_{2}(\theta, \lambda) - \me^{-\mi 2\pi \lambda} D_{2}(\theta + \alpha, \lambda) &= - W_{2}(\theta, \lambda).
	\end{split}
\end{equation*}
However, in the case $\alpha$ is Liouvillean, the first equation, whose small divisor is $|1 - \me^{\mi 2\pi k \alpha} |, \, \forall k \in \mathbb{Z} \setminus \{0\},$ can not be solved at all, even if in the analytic category. Moreover, to ensure $D\in su(1,1),$ more complicated discussions will be needed. Thus, to avoid the homological equation about $W_{1},$ we leave $\mathrm{diag}\{W_{1},\overline{W}_{1}\}$ into
the normal form. As a consequence, from the second step of iteration, we will encounter the variable coefficient homological equations
\begin{equation}\label{homo-eq3}
	(A(\lambda) + V(\theta,\lambda)) \Delta(\theta,\lambda) - \Delta(\theta + \alpha,\lambda) = -U(\theta,\lambda),
\end{equation}
and
\begin{equation}\label{homo-eq4}
	(A(\lambda) + V(\theta,\lambda)) D(\theta,\lambda) - D(\theta+\alpha,\lambda) (A(\lambda) + V(\theta,\lambda)) = -W+\mathrm{diag}\{W_{1},\overline{W}_{1}\},
\end{equation}
where $V=\mathrm{diag}\{V_{1},\overline{V}_{1}\}$ with $\|V\|=O(\varepsilon).$
Note that the map of the system we consider is area-preserving, then
\begin{equation*}
	\det(A+V+W)
	= |\me^{\mathrm{i} 2\pi \lambda}+V_{1} + W_1|^2 - |W_2|^2
	= 1,
\end{equation*}
which implies $|\me^{\mathrm{i} 2\pi \lambda}+V_{1}|\simeq (1+O(\|W\|^2) )^{\frac{1}{2}}$.
Therefore, there exist real-valued functions $\rho,B\in\Lambda_{r}(\mathbb{T}\times \mathcal{O},\mathbb{R})$ with $\|\rho\|=O(\|W\|^2),\,\|B\|=O(\|V\|)$
such that
\begin{equation*}
	\me^{\mathrm{i} 2\pi \lambda} + V_1(\theta,\lambda) = (1+\rho(\theta,\lambda)) \me^{\mathrm{i} 2\pi (\lambda+B(\theta,\lambda))}.
\end{equation*}
See Lemma~\ref{expoB} for details. Hence, $A+V$ could be represented as
\begin{equation*}
	A + V = (1+\rho(\theta,\lambda)) \mathrm{diag} \big\{\me^{\mathrm{i} 2\pi (\lambda+B(\theta,\lambda))},\, \me^{-\mathrm{i} 2\pi (\lambda+B(\theta,\lambda))} \big\}.
\end{equation*}
Thus, if set $\Delta=(\delta,\overline{\delta})^{\mathrm{T}},
-U=(u,\overline{u})^{\mathrm{T}},$ \eqref{homo-eq3} can be rewritten as
\begin{equation}\label{homologicaleq}
	\me^{\mathrm{i} 2\pi (\lambda + B(\theta,\lambda))} \delta(\theta,\lambda) + \rho(\theta,\lambda) \me^{\mathrm{i} 2\pi (\lambda + B(\theta,\lambda))} \delta(\theta,\lambda) - \delta(\theta+\alpha,\lambda) = u(\theta,\lambda).
\end{equation}

To get the desired estimates, we distinguish the discussions into two steps to solve the equation \eqref{homologicaleq}. In the first step, we will kill $B$
by solving the equation
$$\mathcal{B}(\theta+\alpha,\lambda)
-\mathcal{B}(\theta,\lambda)=
-\mathcal{T}_{\overline{Q}_n}B(\theta,\lambda)
+[B(\theta,\lambda)]_{\theta}.$$
Although $B$ is of size $\varepsilon$, $\| \me^{\mathrm{i} 2\pi \mathcal{B}}\|$ could be very large in Liouvillean frequency case. To control it, the trick is to control $\|\mathrm{Im}(\mathcal{B})\|$ at the cost of losing analytic radius greatly, which was first developed in analytic case in \cite{YouZ14}. To get the analytic radius, the key point here is that $\mathcal{T}_{\overline{Q}_n}B$ is a trigonometric polynomial, so is $\mathcal{B}.$ Thus, one can extend $\mathcal{B}$ to a real analytic function with a certain analytic radius (Lemma ~\ref{teq}). Consequently, the width $r$ will go to zero rapidly, and the convergence of the KAM iteration only works in the $C^{\infty} $ topology.

In the second step, we solve the resulting homological equation:
\begin{equation*}
	\me^{\mathrm{i} 2\pi (\lambda+[B(\theta,\lambda)]_{\theta} + \mathcal{R}_{\overline{Q}_n} B(\theta,\lambda))} \delta(\theta,\lambda) + \widetilde{\rho}(\theta,\lambda)\delta(\theta,\lambda) - \delta(\theta+\alpha,\lambda) = u(\theta,\lambda),
\end{equation*}
where $\|\widetilde{\rho}\|=O(\|\rho\|)$.
It can be solved approximately  by  the method of diagonally
dominant, appearing in \cite{WangY17,YouZ14,KrikorianW18} since the norm of both $\mathcal{R}_{\overline{Q}_n}B(\theta,\lambda)$ and $\widetilde{\rho}$ are  sufficiently small. See Proposition~\ref{solveequation} for details.

\subsection{Solving homological equations}
Now we give the details of solving the homological equations.
As stated in subsection~\ref{sec:outline}, we will apply the method in analytic topology to control $\|\mathrm{Im}(\mathcal{B})\|,$ we introduce the analytic norm first.

For the function $f(\theta,\lambda)=\sum_{k\in\mathbb{Z}}\widehat{f}_k(\lambda) \me^{\mathrm{i} 2\pi k \theta}\in \Lambda_r(\mathbb{T}\times \mathcal{O},\mathbb{R})$, define $\widetilde{f}$ by
\begin{equation}\label{defana}
\widetilde{f}(\vartheta,\lambda)=
\sum_{k\in\mathbb{Z}}\widehat{f}_k(\lambda) \me^{\mathrm{i} 2\pi k(\theta+\mathrm{i} \widetilde{\theta})},
\end{equation}
where $\vartheta=\theta+\mi\widetilde{\theta}$ with $\theta\in\TT,\widetilde{\theta}\in\RR.$ We formally define the analytic norm
\begin{equation*}
	\|\widetilde{f}\|^{\omega}_{r,\mathcal{O}} = \sum_{k\in\mathbb{Z}} |\widehat{f}_k|_{\mathcal{O}} \sup_{\theta\in \mathbb{T}, |\widetilde{\theta}|\leq r} |\me^{\mathrm{i} 2\pi k(\theta+\mathrm{i} \widetilde{\theta})}| = \sum_{k\in\mathbb{Z}} |\widehat{f}_k|_{\mathcal{O}} \me^{2\pi |k|r}.
\end{equation*}
If $\mathrm{Im}(\vartheta)=\widetilde{\theta}=0$, $	\widetilde{f}(\vartheta,\lambda)=f(\theta,\lambda)$. Otherwise, for $0<\mathrm{Im}(\vartheta)\leq r$, one has
\begin{equation}\label{analyticnorm}
	\|f\|_{r,\mathcal{O}} = \sum_{k\in\mathbb{Z}} |\widehat{f}_k|_{\mathcal{O}} \me^{\Lambda(2\pi |k|r)}\leq 	\|\widetilde{f}\|^{\omega}_{r,\mathcal{O}}.
\end{equation}
In general,  $\|\widetilde{f}\|^{\omega}_{r,\mathcal{O}}=\infty$. However, if $f$ is a trigonometric polynomial,  $\widetilde{f}$ indeed defines a real analytic function. Motivated by this we have lemma below.

\begin{lemma}\label{teq}
Let $B\in \Lambda_r(\mathbb{T}\times \mathcal{O},\mathbb{R})$ with $r=\overline{Q}_{n-1}^{-2}r_0.$
	Then the equation
	\begin{equation}\label{low}
		\mathcal{B}(\theta+\alpha,\lambda)-	\mathcal{B}(\theta,\lambda)=-\mathcal{T}_{\overline{Q}_{n}}
		B(\theta,\lambda)+[B(\theta,\lambda)]_{\theta}
	\end{equation}
	has a unique solution $\mathcal{B}(\theta,\lambda)$ satisfying the estimate
	\begin{equation*}
		\| \exp\{\mathrm{i} 2\pi\mathcal{B}(\theta,\lambda)\} \|_{\overline{r},\mathcal{O}} \leq \exp\{8\pi^2 r_0 \|B\|_{r, \mathcal{O}}\}, \ \overline{r}=2\overline{Q}_{n}^{-2}r_0.
	\end{equation*}
\end{lemma}
\begin{proof}
	By comparing the Fourier coefficients for the equation \eqref{low}, one has
	\begin{equation*}
		\mathcal{B}(\theta,\lambda) = \sum_{0<|k|<\overline{Q}_{n}}
		\widehat{\mathcal{B}}_k(\lambda) \me^{\mathrm{i} 2\pi k\theta}
		:= \sum_{0<|k|<\overline{Q}_{n}}
		\frac{\widehat{B}_k(\lambda)}{\me^{\mathrm{i} 2\pi k\alpha}-1} \me^{\mathrm{i} 2\pi k\theta}.
	\end{equation*}
	Moreover, for $0<|k|<\overline{Q}_{n},$
	\begin{equation*}
		\big| \widehat{\mathcal{B}}_k \big|_{\mathcal{O}}
		= |\widehat{B}_k|_{\mathcal{O}}|2\sin{\pi k\alpha}|^{-1}
		\leq |\widehat{B}_k|_{\mathcal{O}}(4 \|k\alpha\|_{\mathbb{T}})^{-1} < 2^{-1}\overline{Q}_{n}|\widehat{B}_k|_{\mathcal{O}}.
	\end{equation*}

	For $\theta\in \mathbb{R}$ and $B(\theta,\lambda)\in \mathbb{R}$, one has $\mathcal{B}(\theta,\lambda)\in \mathbb{R}$.
	Therefore, by the definition given by \eqref{defana}, we have
	\begin{equation*}
		\begin{split} \mathrm{Im}\big(\widetilde{\mathcal{B}}
(\vartheta,\lambda)\big)=
\mathrm{Im}\big(\widetilde{\mathcal{B}}
(\vartheta,\lambda)-\mathcal{B}(\theta,\lambda)\big).
		\end{split}
	\end{equation*}
	It follows that
	\begin{equation}\label{im}
		\begin{split}
			\big\| \mathrm{Im}\big(\widetilde{\mathcal{B}}(\vartheta,\lambda)\big) \big\|^{\omega}_{\overline{r},\mathcal{O}}
			&\leq \big\| \widetilde{\mathcal{B}}(\vartheta,\lambda)-\mathcal{B}(\theta,\lambda) \big\|^{\omega}_{\overline{r},\mathcal{O}}\\
			&= \sum_{0<|k|<\overline{Q}_{n}} |\widehat{\mathcal{B}}_k|_{\mathcal{O}} \sup_{\theta\in \mathbb{T},|\widetilde{\theta}| \leq \overline{r}} |\me^{\mathrm{i} 2 \pi k \theta} (\me^{-2\pi k\widetilde{\theta}}-1)|\\
			&\leq 2 \pi \sum_{0<|k|<\overline{Q}_{n}} |\widehat{B}_k|_{\mathcal{O}} \cdot \overline{Q}_{n}|k| \overline{r}  \\
			&\leq 4\pi r_0 \sum_{0<|k|<\overline{Q}_{n}} |\widehat{B}_k|_{\mathcal{O}}\me^{\Lambda(2\pi |k|r)}
			\leq 4\pi r_{0} \|B\|_{r, \mathcal{O}}.
		\end{split}
	\end{equation}
	
	Combing \eqref{analyticnorm} and \eqref{im}, we conclude that
\begin{equation*}
\big\|\me^{\mathrm{i} 2\pi \mathcal{B}(\theta,\lambda)} \big\|_{\overline{r},\mathcal{O}}
		\leq \big\| \me^{\mathrm{i} 2\pi \widetilde{\mathcal{B}}(\vartheta,\lambda)} \big\|^{\omega}_{\overline{r},\mathcal{O}}
		\leq \me^{2\pi \|\mathrm{Im}\widetilde{\mathcal{B}}(\vartheta,\lambda)
\|^{\omega}_{\overline{r},\mathcal{O}}}
\leq \me^{8\pi^2 r_0 \|B\|_{r, \mathcal{O}}}.
\end{equation*}
\end{proof}

\begin{lemma}\label{bettersmalldivisor}
	For $0<\gamma<1,\,\tau>1$ and any $|k|\leq 	K:=\overline{Q}_{n+1}^{\frac{1}{2}},$ one has
	\begin{equation*}
		\big| \me^{\mathrm{i} 2\pi (l \Omega(\lambda)-k\alpha)} - 1 \big| \geq 4 \gamma Q_{n+1}^{-\tau^2}, \, l=1,2
	\end{equation*}
	provided
	\begin{equation*}
		\Omega(\lambda) \in DC_{\alpha}(\gamma,\tau, K,\mathcal{O}) := \Big\{\Omega(\lambda) :
		\begin{array}{l}
			\lambda \in \mathcal{O} \text{ such that }
			\|l \Omega(\lambda)-k\alpha\|_{\mathbb{T}} \geq \frac{\gamma}{(|k|+|l|)^{\tau}},\\
			\forall\, 0<|k|\leq K, \, l=1,2.
		\end{array}
		\Big\}.
	\end{equation*}
\end{lemma}
\begin{proof}
	The similar proof
	can be found in Lemma $3.2$ in \cite{WangY17}.
	To make our paper self-contained, we give a simplified version in Appendix.
\end{proof}

\begin{lemma}\label{trunctes}
	For the function $ f(\theta,\lambda) = \sum_{k\in \mathbb{Z}} \widehat{f}_k(\lambda) \me^{\mathrm{i} 2\pi k\theta}\in \Lambda_r(\mathbb{T}\times \mathcal{O},\mathbb{R}) $, given $K>0$ and $0<\sigma<r,$ we have
	\begin{equation*}
		\|\mathcal{R}_{K} f\|_{r-\sigma,\mathcal{O}} \leq \exp \big\{ -\sigma r^{-1}\Gamma(2\pi K(r-\sigma))\ln (2\pi K(r-\sigma)) \big\} \cdot \|f\|_{r,\mathcal{O}}.
	\end{equation*}
\end{lemma}
\begin{proof}
	Since
	\begin{equation*}
		\begin{split}
			\|\mathcal{R}_{K}f\|_{r-\sigma,\mathcal{O}}
			&= \sum_{|k|\geq K} |\widehat{f}_k|_{\mathcal{O}} \me^{\Lambda(2\pi |k|(r-\sigma))}\\
			&=\sum_{|k|\geq K}|\widehat{f}_k|_{\mathcal{O}} \me^{\Lambda(2\pi |k|r)} \cdot \me^{\Lambda(2\pi |k|(r-\sigma))-\Lambda(2\pi |k|r)}\\
			&\leq \|f\|_{r,\mathcal{O}} \sup_{|k|\geq K} \me^{\Lambda(2\pi |k|(r-\sigma))-\Lambda(2\pi |k|r)},
		\end{split}
	\end{equation*}
	it suffices to estimate $ \sup_{|k|\geq K} \me^{\Lambda(2\pi |k|(r-\sigma))-\Lambda(2\pi |k|r)}$.
	
	By the mean value theorem, for any  $|k| \geq K $, there exists $\xi_{k} \in[r-\sigma, r] $ such that
	\begin{equation*}
		\begin{split}
			\Lambda(2\pi |k|r) &- \Lambda(2\pi |k|(r-\sigma))
			 = \Lambda'(2\pi |k| \xi_{k}) 2\pi |k| \sigma\\
			& = \sigma \xi_{k}^{-1} \Lambda'(2\pi |k| \xi_{k}) 2\pi |k| \xi_{k}
			 = \sigma \xi_{k}^{-1} \Gamma(2\pi |k| \xi_{k}) \ln (2 \pi |k| \xi_{k}) \\
			& \geq \sigma r^{-1} \Gamma(2\pi K(r-\sigma)) \ln (2 \pi K (r-\sigma)).
		\end{split}
	\end{equation*}
	Here we use the fact that $ \forall |k| \geq K,\, 2\pi |k| \xi_{k} \geq 2\pi K(r-\sigma)$ and hypothesis \textbf{(H2)}: $ \Gamma(x) = \frac{x\Lambda'(x)}{\ln x}$ is monotonically increasing on $\RR^{+}$.
	Therefore,
	\begin{equation*}
		\sup_{|k| \geq K} \me^{\Lambda(2\pi |k|(r-\sigma)) - \Lambda(2\pi |k|r)}
		\leq \exp \big\{-\sigma r^{-1} \Gamma(2\pi K(r-\sigma)) \ln (2\pi K(r-\sigma)) \big\}
	\end{equation*}
	by the uniformly upper bound.
	This completes the proof of the estimate.
\end{proof}

We now focus on the homological equation for the unknown function $\delta$:
\begin{equation}\label{homologicaleq1}
	\me^{\mathrm{i} 2\pi l(\lambda+B(\theta,\lambda))} \delta(\theta,\lambda) + b(\theta,\lambda) \delta(\theta,\lambda) - \delta(\theta+\alpha,\lambda) = u(\theta,\lambda),l=1,2.
\end{equation}

\begin{proposition}\label{solveequation}
For $0<\gamma<1, \,\tau>1,\,0<r_{0}<1,$ set
\begin{equation*}
	K=\overline{Q}_{n+1}^{\frac{1}{2}}, \quad
	r=\overline{Q}_{n-1}^{-2}r_0,
	\quad \overline{r}=2\overline{Q}_{n}^{-2}r_0, \quad 0<\sigma<\widetilde{r}<\overline{r}.
\end{equation*}
Let $B\in \Lambda_r(\mathbb{T}\times \mathcal{O},\mathbb{R})$ and $b,\,u\in \Lambda_{\widetilde{r}}(\mathbb{T}\times \mathcal{O},\mathbb{R}),$ Assume
	\begin{equation}\label{bnorm}
		\|B\|_{r,\,\mathcal{O}}\leq \varepsilon_{0}^{\frac{1}{3}}\ll 1,
		\quad \|\mathcal{R}_{\overline{Q}_n}B\|_{\frac{r}{2},\mathcal{O}}
		\leq\gamma^{2}(480 \pi^2 Q_{n+1}^{2\tau^2})^{-1},
	\end{equation}
	\begin{equation}\label{bnorms}
\|b\|_{\widetilde{r},\mathcal{O}}<\gamma^{2}(12 Q_{n+1}^{2\tau^2})^{-1},
	\end{equation}
	and $\widetilde{\lambda}:=\lambda+[B(\theta,\lambda)]_\theta\in DC_{\alpha}(\gamma,\tau,K,\mathcal{O})$, then the homological equation
	\eqref{homologicaleq1} has an approximate solution $\delta$ satisfying
	\begin{equation}\label{so-con}
		\|\delta\|_{\widetilde{r},\mathcal{O}}\leq 32 \gamma^{-2}
		Q_{n+1}^{2\tau^2}
		\|u\|_{\widetilde{r},\mathcal{O}}.
	\end{equation}
Moreover, the error term $\delta^{(er)}$ satisfies
	\begin{equation}\label{er-con}
		\|\delta^{(er)}\|_{\widetilde{r}-\sigma,\mathcal{O}}
		\leq 16 \exp \big\{ -\sigma\widetilde{r}^{-1} \Gamma(2\pi K(\widetilde{r}-\sigma)) \ln (2\pi K(\widetilde{r}-\sigma)) \big\} \cdot \|u\|_{\widetilde{r},\mathcal{O}}.
	\end{equation}
\end{proposition}
\begin{proof}
	Let
	\begin{equation*}
		\widetilde{\delta}(\theta,\lambda) = \me^{\mathrm{i} 2\pi l\mathcal{B}(\theta,\lambda)} \delta(\theta,\lambda), \qquad
		\widetilde{u}(\theta,\lambda) = \me^{\mathrm{i} 2\pi l\mathcal{B}(\theta,\lambda)} u(\theta,\lambda),
	\end{equation*}
	where $ \mathcal{B} $ is the solution to
	\begin{equation*}
		\mathcal{B}(\theta+\alpha,\lambda)-	\mathcal{B}(\theta,\lambda)=-\mathcal{T}_{\overline{Q}_{n}}
		B(\theta,\lambda)+[B(\theta,\lambda)]_{\theta}.
	\end{equation*}
	Then, equation \eqref{homologicaleq1} becomes
	\begin{equation}\label{homologicaleq2}
		\me^{\mathrm{i} 2\pi l\widetilde{\lambda}} \cdot \widetilde{\delta}(\theta,\lambda) + \widetilde{b}(\theta,\lambda) \cdot \widetilde{\delta}(\theta,\lambda) - \widetilde{\delta}(\theta+\alpha,\lambda) = \widetilde{\widetilde{u}}(\theta,\lambda),
	\end{equation}
	where
	\begin{equation*}
		\begin{split}	
			\widetilde{b}(\theta,\lambda)
			& = \big(\me^{\mathrm{i} 2\pi l\mathcal{R}_{\overline{Q}_n}B}-1 \big) \cdot \me^{\mathrm{i} 2\pi l\widetilde{\lambda}} +b(\theta,\lambda) \me^{\mathrm{i} 2\pi l (-\mathcal{T}_{\overline{Q}_{n}}
				B+[B]_{\theta})},\\
			\widetilde{\widetilde{u}}(\theta,\lambda)
			& = \widetilde{u}(\theta,\lambda)
			\me^{\mathrm{i} 2\pi l (-\mathcal{T}_{\overline{Q}_{n}} B+[B]_{\theta})}.
		\end{split}
	\end{equation*}
It follows from \eqref{bnorm} and \eqref{bnorms} that
	\begin{equation}\label{newnotationes}
		\begin{split}	
			\|\widetilde{b}\|_{\widetilde{r},\mathcal{O}}
			& \leq 5 \pi \big\| \me^{\mathrm{i} 2\pi l\mathcal{R}_{\overline{Q}_n}B}-1\big\|_{\widetilde{r},\mathcal{O}}
			+  \|b(\theta,\lambda)\|_{\widetilde{r},\mathcal{O}} \big\| \me^{\mathrm{i} 2\pi l (-\mathcal{T}_{\overline{Q}_{n}} B+[B]_{\theta})}\big\|_{\widetilde{r},\mathcal{O}}\\
			& \leq 40 \pi^2 \|\mathcal{R}_{\overline{Q}_n}B\|_{\frac{r}{2},\mathcal{O}} +2\|b(\theta,\lambda)\|_{\widetilde{r},\mathcal{O}}
			\leq \frac{\gamma^{2}}{4Q_{n+1}^{2\tau^2}} ,
		\end{split}
	\end{equation}
	and
	\begin{equation*}
	\big\| \widetilde{\widetilde{u}} \big\|_{\widetilde{r},\mathcal{O}}
		\leq \|\widetilde{u}\|_{\widetilde{r},\mathcal{O}} \cdot
		\big\| \me^{\mathrm{i} 2\pi l (-\mathcal{T}_{\overline{Q}_{n}} B+[B]_{\theta})}\big\|_{\widetilde{r},\mathcal{O}}\leq 2\|\widetilde{u}\|_{\widetilde{r},\mathcal{O}}.
	\end{equation*}

	We now solve the truncation version of the equation
	\eqref{homologicaleq2}:
	\begin{equation}\label{truncationequation}
		\mathcal{T}_{K} \big( \me^{\mathrm{i} 2\pi l\widetilde{\lambda}} \cdot \widetilde{\delta}(\theta,\lambda) + \widetilde{b}(\theta,\lambda) \cdot \widetilde{\delta}(\theta,\lambda) - \widetilde{\delta}(\theta+\alpha,\lambda)\big)
		= \mathcal{T}_{K} \widetilde{\widetilde{u}}(\theta,\lambda).
	\end{equation}
	Assume
	\begin{equation*}
		\widetilde{\delta}(\theta,\lambda) = \sum_{k\in \mathbb{Z},|k|<K} \widehat{\widetilde{\delta}}_{k}(\lambda) \me^{\mathrm{i} 2\pi k\theta},\
	\end{equation*}
	and expand $ \widetilde{b} $ and $ \widetilde{\widetilde{u}} $ into Fourier series
	\begin{equation*}
		\widetilde{b}(\theta,\lambda) = \sum_{k\in \mathbb{Z}} \widehat{\widetilde{b}}_{k}(\lambda) \me^{\mathrm{i} 2\pi k\theta},
		\qquad  \widetilde{\widetilde{u}}(\theta,\lambda) = \sum_{k\in \mathbb{Z}} \widehat{\widetilde{\widetilde{u}}}_{k}(\lambda) \me^{\mathrm{i} 2\pi k\theta}.
	\end{equation*}
	Comparing the Fourier coefficients of the equation \eqref{truncationequation},
	for $|k|<K,$ we obtain
	\begin{equation*}
		(\me^{\mathrm{i} 2\pi l\widetilde{\lambda}} - \me^{\mathrm{i} 2\pi k\alpha}) \cdot \widehat{\widetilde{\delta}}_{k}(\lambda) + \sum_{|k_{1}|<K} \widehat{\widetilde{b}}_{k-k_{1}}(\lambda) \widehat{\widetilde{\delta}}_{k_{1}}(\lambda)
		=\widehat{\widetilde{\widetilde{u}}}_{k}(\lambda).
	\end{equation*}
	This can be viewed as the following linear equations:
	\begin{equation}\label{matrixeq}
		(\mathcal{S}+\mathcal{P})\mathfrak{F}=\mathfrak{U},
	\end{equation}
	where
	\begin{equation*}
		\begin{split}
			\mathcal{S}
			& = \mathrm{diag} \left\{\me^{\mathrm{i} 2\pi k\alpha} \big( \me^{\mathrm{i} 2\pi (l\widetilde{\lambda}-k\alpha)}-1\big)\right\}_{|k|<K},\\
			\mathcal{P}
			& = \left(\widehat{\widetilde{b}}_{k_{1}-k_{2}}(\lambda)\right)_{|k_{1}|,|k_{2}|<K},\\
			\mathfrak{F}
			& = \left(\widehat{\widetilde{\delta}}_{k}(\lambda)\right)_{|k|<K}^{\mathrm{T}}, \qquad \mathfrak{U}=
\left(\widehat{\widetilde{\widetilde{u}}}_{k}
(\lambda)\right)_{|k|<K}^{\mathrm{T}}.
		\end{split}
	\end{equation*}

Set
	\begin{equation*}
		E_{\widetilde{r}}=\mathrm{diag} \left\{ \me^{\Lambda(2\pi |k| \widetilde{r})} \right\}_{|k|<K}.
	\end{equation*}
	Then, the equations \eqref{matrixeq} can be rewritten as
	\begin{equation*}
		\big( \mathcal{S} + E_{\widetilde{r}} \mathcal{P} E_{\widetilde{r}}^{-1} \big) E_{\widetilde{r}} \mathfrak{F}
		= E_{\widetilde{r}} \mathfrak{U}.
	\end{equation*}
	It follows from  $\widetilde{\lambda}\in DC_{\alpha}(\gamma,\tau,K,\mathcal{O})$ and
	Lemma~\ref{bettersmalldivisor} for small divisors that
	\begin{equation*}
		\|\mathcal{S}^{-1}\|_{\mathcal{O}}
		= \max_{|k|<K} \sup_{\lambda\in\mathcal{O}} \left\{\frac{1}{\big| \me^{\mathrm{i} 2\pi (l\widetilde{\lambda}-k\alpha)} - 1\big|} +
		\frac{2\pi l \big| 1 + \frac{\partial [B(\theta,\lambda)]_{\theta}}{\partial\lambda} \big|} {\big| \me^{\mathrm{i} 2\pi (l\widetilde{\lambda} - k\alpha)}-1\big|^2} \right\}
		\leq 2 \gamma^{-2}Q_{n+1}^{2\tau^2},
	\end{equation*}
where $\|\cdot\|_{\mathcal{O}}$ is defined by
	\begin{equation*}
		\| M \|_{\mathcal{O}}
		=\max_{i}\sum_{j}|a_{ij}|_{\mathcal{O}}
	\end{equation*}
	with $ a_{ij} $ being the $ (i,j) $-th element of the matrix $ M $. Similarly,
	\begin{equation*}
		\begin{split}
			\big\| E_{\widetilde{r}} \mathcal{P} E_{\widetilde{r}}^{-1} \big\|_{\mathcal{O}}
			& = \max_{|k_1|<K} \sum_{|k_2|<K} \big| \me^{\Lambda(2\pi |k_1| \widetilde{r})} \cdot \widehat{\widetilde{b}}_{k_{1}-k_{2}}(\lambda) \cdot \me^{-\Lambda(2\pi |k_2|\widetilde{r})}\big|_{\mathcal {O}}\\
			& \leq \max_{|k_1|<K}\sum_{|k_2|<K}\big|\me^{\Lambda ( 2\pi ||k_1|-|k_2||\widetilde{r} )} \cdot \widehat{\widetilde{b}}_{k_{1}-k_{2}}(\lambda)\big|_{\mathcal {O}}\\
			& \leq \max_{|k_1|<K}\sum_{|k_2|<K}\big| \me^{\Lambda( 2\pi |k_1-k_2|\widetilde{r} )}\cdot\widehat{\widetilde{b}}_{k_{1}-k_{2}}(
			\lambda)\big|_{\mathcal {O}}
			 \leq\|\widetilde{b}\|_{\widetilde{r}, \mathcal{O}},
		\end{split}
	\end{equation*}
where the first and the second inequalities are from the subadditivity assumption on the function $\Lambda$ (see $(\textbf{H1})$ in Theorem~\ref{mainthm}) and the fact that $\Lambda$ is non-decreasing. Therefore,  by \eqref{newnotationes}, one has
	\begin{equation*}
		\big\| \mathcal{S}^{-1}E_{\widetilde{r}}\mathcal{P}E_{\widetilde{r}}^{-1} \big\|_{\mathcal{O}}
		\leq 2 \gamma^{-2} Q_{n+1}^{2\tau^2}  \|\widetilde{b}\|_{\widetilde{r},\mathcal{O}} \leq 2^{-1}.
	\end{equation*}
	This implies that  $\mathcal{S}+E_{\widetilde{r}}\mathcal{P}E_{\widetilde{r}}^{-1}$ has a bounded inverse, i.e.,
	\begin{equation*}
		\begin{split}
			\big\| (\mathcal{S}+E_{\widetilde{r}}\mathcal{P}E_{\widetilde{r}}^{-1})^{-1} \big\|_{\mathcal{O}}
			&= \big\| \big(I+\mathcal{S}^{-1}E_{\widetilde{r}}\mathcal{P}E_{\widetilde{r}}^{-1}\big)^{-1}\mathcal{S}^{-1} \big\|_{\mathcal{O}}\\ &\leq\frac{1}{1-\|\mathcal{S}^{-1}E_{\widetilde{r}}\mathcal{P}E_{\widetilde{r}}^{-1}\|_{\mathcal{O}}}\|\mathcal{S}^{-1}\|_{\mathcal{O}}\\
			&\leq 4 \gamma^{-2}Q_{n+1}^{2\tau^2}.
		\end{split}
	\end{equation*}
	As a consequence,
	\begin{equation*}
		\begin{split}
			\|\widetilde{\delta}\|_{\widetilde{r},\mathcal{O}}
			& = \sum_{|k|<K} \big|\widehat{\widetilde{\delta}}_{k} \big|_{\mathcal {O}} \me^{\Lambda(2\pi |k|\widetilde{r})}
			 = \|E_{\widetilde{r}}\mathfrak{F}\|_{\mathcal {O}}\\
			& = \|(\mathcal{S}+E_{\widetilde{r}}\mathcal{P}E_{\widetilde{r}}^{-1})^{-1}E_{\widetilde{r}}\mathfrak{U}\|_{\mathcal {O}} \leq 4 \gamma^{-2} Q_{n+1}^{2\tau^2} \|\widetilde{\widetilde{u}}\|_{\widetilde{r},\mathcal{O}}.
		\end{split}
	\end{equation*}
	Going back to
	$\delta(\theta,\lambda) = \me^{-\mathrm{i} 2\pi l\mathcal{B}(\theta,\lambda)}\widetilde{\delta}(\theta,\lambda)$ and combining with Lemma~\ref{teq}, one has
	\begin{equation}\label{so-es}
		\|\delta\|_{\widetilde{r},\mathcal{O}} \leq \|\me^{-\mathrm{i} 2\pi l\mathcal{B}(\theta,\lambda)}\|_{\widetilde{r},\mathcal{O}} \|\widetilde{\delta}\|_{\widetilde{r},\mathcal{O}}
		\leq 32 \gamma^{-2}Q_{n+1}^{2\tau^2} \|u\|_{\widetilde{r},\mathcal{O}}.
	\end{equation}

	In addition, one can verify that the error term $\delta^{(er)}$ is
	\begin{equation*}
		\delta^{(er)} = \me^{-\mathrm{i} 2\pi l\mathcal{B}(\theta,\lambda)}\cdot\mathcal{R}_{
			K}\big( \widetilde{b}(\theta,\lambda) \widetilde{\delta}(\theta,\lambda)-\widetilde{\widetilde{u}}(\theta,\lambda)\big).
	\end{equation*}
	By Lemma~\ref{teq} and Lemma~\ref{trunctes}, one has	
	\begin{equation*}
		\begin{split}
			\|\delta^{(er)}\|_{\widetilde{r}-\sigma,\mathcal{O}}
			&\leq 2 \me^{-\sigma\widetilde{r}^{-1} \Gamma(2\pi K(\widetilde{r}-\sigma))\ln(2\pi K(\widetilde{r}-\sigma))} \big(\|\widetilde{b}\|_{\widetilde{r},\mathcal{O}}\|\widetilde{\delta}\|_{\widetilde{r},\mathcal{O}}+	\|\widetilde{\widetilde{u}}\|_{\widetilde{r},\mathcal{O}}\big)\\
			& \leq 16 \exp \big\{  -\sigma\widetilde{r}^{-1}\Gamma(2\pi K(\widetilde{r}-\sigma))\ln(2\pi K(\widetilde{r}-\sigma)) \big\} \cdot
			\|u\|_{\widetilde{r},\mathcal{O}}.
		\end{split}
	\end{equation*}
\end{proof}

\section{KAM step: proof of the KAM Theorem~\ref{kamthm}} \label{sec:proofkam}
In this section, we present the proof of Theorem~\ref{kamthm} by constructing a modified KAM induction. The philosophy is to construct an infinite series of area-preserving coordinates transformations to make the perturbation smaller and smaller at the cost of excluding a small set of parameters and losing $r$ greatly.

\subsection{Main iterative lemma}\label{sec:infiniteinte}
For $\tau>1,$ set $\mathbb{A}>18$ such that
\begin{equation*}
		\begin{split}
 \mathbb{A}^{4} \geq (4 \tau^{2})^{-1}(1 + 2 \ln (48 c)),
		\end{split}
	\end{equation*}
where $ c $ is the largest constant independent of the KAM scheme.
By the assumption (\textbf{H2}),
we know that there exists $\widetilde{T}>0 $ such that for any $x\geq \widetilde{T}$, one has $ \Gamma(x) \geq 324 \mathbb{A}^8 \tau^4 $.
Denote
\begin{equation*}
	T=\max \big\{ 20r^{-1},\,\, 48c\gamma^{-2},\,\, \widetilde{T}^{3} \big\}.
\end{equation*}
Once we fix $\mathbb{A}$ and $T$ above, the estimates and discussions in Section~\ref{sec:1proofkam} will hold.

For $T$ defined above, we claim that there exists  $n_{0}\in \mathbb{N}$ such that $Q_{n_{0}+1}\leq T^{\mathbb{A}^4}$ and $\overline{Q}_{n_{0}+1}\geq T$. Indeed, let $m_0$ be such that $Q_{m_{0}}\leq T\leq Q_{m_{0}+1}$. If $\overline{Q}_{m_0}\geq T$, we take $n_0=m_0-1$. Otherwise, if $\overline{Q}_{m_0}\leq  T$, by the definition of $\{Q_k\}$, $Q_{m_{0}+1}\leq T^{\mathbb{A}^4}$ holds. Then, we choose $n_0=m_0$. In the following, we shorten  $n_0$ as $0$. Namely, $\overline{Q}_n$ stands for $\overline{Q}_{n+n_0}$.

Since $Q_{n_{0}+1}\leq T^{\mathbb{A}^4}$, we choose $ \varepsilon > 0 $ sufficiently small, depending on the constants $0<r,s,\gamma<1$ and $\tau>1,$ but not on $\alpha$, such that
\begin{equation}\label{in-s}
	\varepsilon \leq \min \big\{
	\{(16 \pi^{2})^{-1}\ln 2\}^{3},\,\,
	T^{-18 \mathbb{A}^{4} \tau^2},\,\,
	\{(240)^{-1}s\}^{3} \big\}.
\end{equation}
Then we define the main iterative sequences by
\begin{equation}\label{definitionparameters}
	\begin{split}
\varepsilon _{0} &=\varepsilon,\ \ \gamma_{0}=\gamma, \ \ s_{0}= s,\ \ r_{0}= r,\ \ \eta_{n}= (n+2)^{-2},\\
		s_{n+1} &= s_n (1-\eta_{n}),\quad
\gamma_{n}=\gamma_{0}\eta_{n},\quad r_{n+1} = \overline{Q}_{n}^{-2}r_0,\\
\mathcal{E}_{n+1}& = \overline{Q}_{n+1}^{-\Gamma^{\frac{1}{2}}
(\overline{Q}_{n+1}^{\frac{1}{3}})}, \quad
		\varepsilon_{n+1}=\mathcal{E}_{n+1}\varepsilon_{n},\quad
		K_{n} = \overline{Q}_{n+1}^{\frac{1}{2}}.
	\end{split}
\end{equation}

\begin{lemma}[Iterative Lemma]\label{iterationlemma}
Suppose that $\varepsilon$ satisfies \eqref{in-s} and consider the equation for the area-preserving system
	\begin{equation} \label{202207182}
\begin{split}
(Eq)_n:\ \	
		X_n(\theta+\alpha,\lambda)
		& = (A(\lambda) + V_n(\theta,\lambda)) X_n(\theta,\lambda) + U_n(\theta,\lambda)  \\
		& \qquad +W_n(\theta,\lambda) X_n(\theta,\lambda)+ R_n(X_n(\theta,\lambda), \theta, \lambda)
\end{split}
\end{equation}
	on the unknown function $X_n\in \Lambda_{r_n}(\mathbb{T}\times \mathcal{O}_{n-1}, \mathcal{C}^{2}),$ where
\begin{equation}\label{202207200}
\begin{split}
V_{n}& = \mathrm{diag}\{G_{n}, \overline{G}_{n}\},
G_{n}\in \Lambda_{r_{n}}(\mathbb{T}\times \mathcal{O},\mathbb{C}),
U_n
\in\Lambda_{r_{n}}(\mathbb{T}\times \mathcal{O},\mathcal{C}^{2}),\\
W_n&\in\Lambda_{r_{n}}(\mathbb{T}\times \mathcal{O},\mathcal{C}^{2\times2}),R_{n}
\in\Lambda_{r_{n}}(\mathcal{C}^{2}\times\mathbb{T}\times \mathcal{O},\mathcal{C}^{2}),
\end{split}
\end{equation}
satisfies
	\begin{equation}\label{bestimate}
		\begin{split}
			&\|V_n-V_{n-1}\|_{r_n,\mathcal{O}_{n-1}}\leq 3\varepsilon_{n-1}^{\frac{1}{2}},
			\quad \|U_n\|_{r_n, \mathcal{O}_{n-1}}\leq \varepsilon_n,
			\quad \|W_n\|_{r_n, \mathcal{O}_{n-1}}\leq \varepsilon_n^{\frac{1}{2}},\\
			&R_n(0, \theta, \lambda) = 0,
			\quad \partial_{1} R_n(0, \theta, \lambda) = 0,
			\quad  \|\partial^2_{11}R_n\|_{r_n,s_n,\mathcal{O}_{n-1}}\leq \prod_{l=0}^{n-1} \me^{24 \varepsilon_{l}^{\frac{1}{3}}},
		\end{split}
	\end{equation}
where we set $V_{0}=0$ and $\mathcal{O}_{n-1}$ is defined by
	\begin{equation*}
		\mathcal{O}_{n-1} = \Big\{ \lambda\in\mathcal{O}:
		\begin{array}{l}
			\big\| l(\lambda + 	[B_{n-1}(\theta,\lambda)]_{\theta})
			-k\alpha \big\|_{\mathbb{T}} \geq 	\frac{\gamma_{n-1}}{(|k|+|l|)^{\tau}},\\
			\forall \,0<|k|<K_{n-1},\, l=1,2.
		\end{array}
		\Big\},
	\end{equation*}
here $B_{n-1}$ is the one given by \eqref{eq} in Lemma~\ref{expoB}  with $G_{n-1}$ in place of $G.$ Moreover, set $B_{n}$ be the one given by \eqref{eq} in Lemma~\ref{expoB}  with $G_{n}$ in place of $G.$
	Then there exist a subset $\mathcal{O}_{n}\subseteq\mathcal{O}_{n-1}$ with
	\begin{equation}\label{ta-out}
\mathcal{O}_{n}=\mathcal{O}_{n-1}\setminus\bigcup_{K_{n-3}\leq |k|< K_{n}}g_{k}^{n}(\gamma _{n}),
	\end{equation}
	where $K_{-2}=K_{-1}=0$ and
	\begin{equation}\label{out}
		g_{k}^{n}(\gamma _{n}) = \Big\{ \lambda\in\mathcal{O}_{n-1}:
		\big\| l(\lambda+[B_{n}(\theta,\lambda)]_{\theta})-k\alpha\big\|_{\mathbb{T}}
		< \frac{\gamma_{n}}{(|k|+|l|)^{\tau}},\, l=1,2\Big\},
	\end{equation}
	and an area-preserving transformation
	\begin{equation*}
		\Phi_{n}: \Lambda_{r_{n+1}}(\mathbb{T}\times \mathcal{O}_{n}, \mathcal{C}^{2}) \rightarrow \Lambda_{r_n}(\mathbb{T}\times \mathcal{O}_{n}, \mathcal{C}^{2})
	\end{equation*}
	of the form
	\begin{equation}\label{nform}
		X_{n} = \me^{D_{n+1}} X_{n+1} + \Delta_{n+1},
\end{equation}
where
\begin{equation}\label{nforms}
D_{n+1}\in\Lambda_{r_{n+1}}(\mathbb{T}\times \mathcal{O}_{n}, su(1,1)),\quad \Delta_{n+1}\in\Lambda_{r_{n+1}}(\mathbb{T}\times \mathcal{O}_{n}, \mathcal{C}^{2})
	\end{equation}
with the estimates
\begin{equation} \label{202207183}
		\|D_{n+1}\|_{r_{n+1}, \mathcal{O}_n}
		\leq 4\varepsilon_n^{\frac{1}{3}},
		\qquad \|\Delta_{n+1}\|_{r_{n+1}, \mathcal{O}_n}
		\leq 4\varepsilon_n^{\frac{5}{6}},
	\end{equation}
such that $\Phi_{n}$ changes $(Eq)_n$ defined by \eqref{202207182} into $(Eq)_{n+1}$, which has the analogous form of \eqref{202207182} and satisfies the estimates in \eqref{bestimate} with $(n+1)$ in place of $n$.
\end{lemma}
The proof of Lemma~\ref{iterationlemma} constitutes of a sequence of subiterations, which make the perturbation as small as we need.

\subsection{A finite induction}\label{sec:finite-induction}
For the parameters defined by \eqref{definitionparameters}, set $\widetilde{\varepsilon}_{0}= \varepsilon_{n},\widetilde{\varepsilon }_{j+1} = \me^{-1} \widetilde{\varepsilon }_{j},j\geq0,$ and let $L\in\NN$ such that
$\widetilde{\varepsilon}_{L}\leq\varepsilon_{n+1}< \widetilde{\varepsilon}_{L-1} $, i.e.,
\begin{equation}\label{L} \Gamma^{\frac{1}{2}}(\overline{Q}_{n+1}^{\frac{1}{3}})
\ln\overline{Q}_{n+1}\leq L<1+\Gamma^{\frac{1}{2}}(\overline{Q}_{n+1}^{\frac{1}{3}})
\ln\overline{Q}_{n+1}.
\end{equation}
Set $\sigma=(2L)^{-1},K = \overline{Q}_{n+1}^{\frac{1}{2}}$ and define the sequences below.
\begin{equation}\label{finite-se}
\begin{split}
\widetilde{r}_{0}= 2 \overline{Q}_{n}^{-2} r_0,\ \ \widetilde{s}_{0} = s_{n},\ \
\widetilde{r}_{j+1} = \widetilde{r}_{j} - \sigma \widetilde{r}_{0},\
\
\widetilde{s}_{j+1} = \widetilde{s}_{j} - \eta_{n}\eta_{j}\widetilde{s}_{0}.
	\end{split}
\end{equation}
\begin{lemma}\label{parametersestimates}
Set $\mathcal{E}=\overline{Q}_{n}^{-\Gamma^{\frac{1}{2}}(\overline{Q}_{n}^{\frac{1}{3}})}$, and assume
\begin{equation*}
\widetilde{\varepsilon}_0 \leq \mathcal{E}, \qquad \Gamma(\overline{Q}_{n}^{\frac{1}{3}})\geq 324 \mathbb{A}^8\tau^4.
	\end{equation*}
Then we have
\begin{equation}\label{ddefinitionparameters0}
\begin{split}
\widetilde{\varepsilon}_{0}\leq Q_{n+1}^{-18\tau^2},\quad
\exp \big\{ -\sigma\Gamma(\overline{Q}_{n+1}^{\frac{1}{3}})
\ln(\overline{Q}_{n+1}^{\frac{1}{3}}) \big\}
		\leq \me^{-2\mathbb{A}^4\tau^2}.
\end{split}
\end{equation}
\end{lemma}
\begin{proof}
The first inequality in \eqref{ddefinitionparameters0} is induced by
	\begin{equation*}
		\widetilde{\varepsilon}_0\leq \mathcal{E}=\overline{Q}_{n}^{-\Gamma^{\frac{1}{2}}(\overline{Q}_{n}^{\frac{1}{3}})}\leq \overline{Q}_{n}^{-18\mathbb{A}^4\tau^2}
		\leq  Q_{n+1}^{-18\tau^2},
	\end{equation*}
	where the last inequality is from $Q_{n+1}\leq \overline{Q}_{n}^{\mathbb{A}^4}$ in
Lemma~\ref{bridgeestimate}.
	
Now we consider the second inequality in \eqref{ddefinitionparameters0}.
	By the definition of $\sigma $ in \eqref{finite-se} and the fact that $\Gamma$ is non-decreasing, one has
	\begin{equation*}
		\sigma \Gamma(\overline{Q}_{n+1}^{\frac{1}{3}}) \ln(\overline{Q}_{n+1}^{\frac{1}{3}})
		= (2L)^{-1} \Gamma(\overline{Q}_{n+1}^{\frac{1}{3}}) \ln(\overline{Q}_{n+1}^{\frac{1}{3}})
		\geq 8^{-1} \Gamma^{\frac{1}{2}}(\overline{Q}_{n+1}^{\frac{1}{3}})
		\geq 2\mathbb{A}^4\tau^2.
	\end{equation*}
\end{proof}

Now we give the following iterative lemma for a finite induction to achieve one step of KAM scheme.

\begin{lemma}\label{finite}
Suppose that the assumptions in Lemma~\ref{parametersestimates} are satisfied and consider the equations for area-preserving systems
\begin{align} \label{eqj}
(\widetilde{Eq})_j:\ \		
\widetilde{X}_j(\theta+\alpha,\lambda)
		& = (A(\lambda) + V(\theta,\lambda) + v_j(\theta,\lambda)) \widetilde{X}_j(\theta,\lambda) + \widetilde{U}_j(\theta,\lambda) \\
		\nonumber
		& \qquad \quad +\widetilde{W}_j(\theta,\lambda)\widetilde{X}_j(\theta,\lambda) +\widetilde{R}_j(\widetilde{X}_{j}(\theta,\lambda),\theta, \lambda),
	\end{align}
	where, by setting $\mathcal{O}=\mathcal{O}_{n}$ with $\mathcal{O}_{n}$ being the one in Lemma~\ref{iterationlemma},
	\begin{equation}\label{202207190}
\begin{split}
 v_{j}& = \mathrm{diag}\{g_{j}, \overline{g}_{j}\},\ \
		 g_{j}  \in \Lambda_{\widetilde{r}_{j}}(\mathbb{T}\times \mathcal{O},\mathbb{C}),\ \widetilde{U}_j
\in\Lambda_{\widetilde{r}_{j}}(\mathbb{T}\times \mathcal{O},\mathcal{C}^{2}),\\
\widetilde{W}_j&\in\Lambda_{\widetilde{r}_{j-1}}(\mathbb{T}\times \mathcal{O},\mathcal{C}^{2\times2}),\widetilde{R}_j
\in\Lambda_{\widetilde{r}_{j}}(\mathcal{C}^{2}\times\mathbb{T}\times \mathcal{O},\mathcal{C}^{2}),
\end{split}
\end{equation}
and $V:=\mathrm{diag}\{G,\overline{G}\}$ is the abbreviated form of $V_{n},$ which is the one in Lemma~\ref{iterationlemma}, with the estimates
\begin{equation}\label{202207191}
\begin{split}
& \|V\|_{r_{n},\mathcal{O}} \leq \varepsilon_{0}^{\frac{1}{3}},
			\quad
			\|\widetilde{U}_j\|_{\widetilde{r}_{j}, \mathcal{O}} \leq \widetilde{\varepsilon}_{j},
			\quad
\|\widetilde{W}_j\|_{\widetilde{r}_{j}, \mathcal{O}} \leq \widetilde{\varepsilon}_{j}^{\frac{1}{2}},\\
&\|v_{j}\|_{\widetilde{r}_{j-1},\mathcal{O}} \leq \sum_{m=0}^{j-1} \widetilde{\varepsilon}_{m}^{\frac{1}{2}},\quad \widetilde{R}_{j}(0, \theta, \lambda) = 0,
\quad
\partial_{1} \widetilde{R}_{j}(0, \theta, \lambda) = 0,\\
&\|\partial^2_{11} \widetilde{R}_{j}\|_{\widetilde{r}_{j},\widetilde{s}_{j},\mathcal{O}} \leq \Big(\prod_{m=0}^{j-1}(1 + 6 \widetilde{\varepsilon}_{m}^{\frac{1}{3}})\Big) \|\partial_{11}^2 \widetilde{R}_{0}\|_{\widetilde{r}_{0}, \widetilde{s}_{0}, \mathcal{O}}.
		\end{split}
	\end{equation}
Then, there is an area-preserving transformation
	\begin{equation*}
		\widetilde{\Phi}_{j}:  \Lambda_{\widetilde{r}_{j+1}}(\mathbb{T}\times \mathcal{O}, \mathcal{C}^{2}) \rightarrow  \Lambda_{\widetilde{r}_{j}}(\mathbb{T}\times \mathcal{O}, \mathcal{C}^{2})
	\end{equation*}
	of the form
	\begin{equation}\label{trans}
		\widetilde{X}_{j} (\theta, \lambda)
		= \me^{\widetilde{D}_{j}(\theta, \lambda)} \widetilde{X}_{j+1} (\theta, \lambda)
		+ \widetilde{\Delta}_{j} (\theta, \lambda)
	\end{equation}
	where
\begin{equation*}
\begin{split}
\widetilde{D}_{j}
&= \Big(\begin{array}{cc}
		0  &  \widetilde{d}_{j} \\
		\overline{\widetilde{d}_{j}} &   0
	\end{array}
	\Big)\in\Lambda_{\widetilde{r}_{j+1}}(\mathbb{T}\times \mathcal{O}, su(1,1)),\quad\|\widetilde{D}_j\|_{\widetilde{r}_{j}, \mathcal{O}}
		\leq\widetilde{\varepsilon}_{j}^{\frac{1}{3}},\\
\widetilde{\Delta}_j &= \big(\widetilde{\delta}_j,\overline{\widetilde{\delta}}_j \big)^{\mathrm{T}}\in\Lambda_{\widetilde{r}_{j+1}}(\mathbb{T}\times \mathcal{O}, \mathcal{C}^{2}),\quad\|\widetilde{\Delta}_j\|_{\widetilde{r}_{j}, \mathcal{O}}
		\leq\widetilde{\varepsilon}_{j}^{\frac{5}{6}},
\end{split}
	\end{equation*}
such that $\widetilde{\Phi}_{j}$ changes
$(\widetilde{Eq})_j$ defined by \eqref{eqj} into a new one, which has the analogous  expansion of
$(\widetilde{Eq})_j$ and satisfies the estimates in \eqref{202207190} and \eqref{202207191} with $(j+1)$ in place of $j.$
\end{lemma}
Now, we give the proof of Lemma~\ref{finite}, which is separated into three parts.

	$ \spadesuit: $ \textbf{The homological equations.}
	By the idea presented in  Section~\ref{sec:outline}, we split $\widetilde{W}_{j}$	into two parts:
	\begin{equation*}
		\widetilde{W}_{j}
		=\left(\begin{array}{cc}
			\widetilde{W}_{j1}  & \widetilde{W}_{j2}
			\medskip \\
			\overline{\widetilde{W}}_{j2}  &  \overline{\widetilde{W}}_{j1}
		\end{array}
		\right)
		=\left(\begin{array}{cc}
			\widetilde{W}_{j1}  & 0
			\medskip \\
			0 & \overline{\widetilde{W}}_{j1}
		\end{array}
		\right)
		+\left(\begin{array}{cc}
			0  & \widetilde{W}_{j2}
			\medskip \\
			\overline{\widetilde{W}}_{j2} &  0
		\end{array}
		\right)
		=:\mathbb{W}_{j1} + \mathbb{W}_{j2}.
	\end{equation*}
	Let
	\begin{equation*}
		v_{j+1} = v_j + \mathbb{W}_{j1}
		:= \left(
			\begin{array}{cc}
				g_{j+1} & 0\\
				0 & \overline{g}_{j+1}
			\end{array}
		\right),
	\end{equation*}
	i.e., $ g_{j+1} = g_{j} + \widetilde{W}_{j1} $.
	Substituting $ \widetilde{X}_{j} (\theta, \lambda)
	= \me^{\widetilde{D}_{j}(\theta, \lambda)} \widetilde{X}_{j+1} (\theta, \lambda)
	+ \widetilde{\Delta}_{j} (\theta, \lambda) $ into \eqref{eqj}
we will get a equation like \eqref{kam-eq2}. With more calculations we obtain
\begin{equation}\label{j1}
		\begin{split}
			\widetilde{X}_{j+1}(\theta + \alpha)
			& =  (A + V(\theta) + v_{j+1}(\theta)) \widetilde{X}_{j+1}(\theta) + \Big[ (A + V(\theta) + v_{j+1}(\theta)) \widetilde{D}_{j} (\theta)  \\
			& - \widetilde{D}_{j} (\theta+\alpha) (A + V(\theta) + v_{j+1}(\theta)) + \mathbb{W}_{j2} (\theta) \Big] \widetilde{X}_{j+1}(\theta) \\
			& + \me^{-\widetilde{D}_{j}(\theta + \alpha)} \Big[ (A + V(\theta) + v_{j+1}(\theta)) \widetilde{\Delta}_{j}(\theta) + \widetilde{U}_{j}(\theta) - \widetilde{\Delta}_{j}(\theta+\alpha) \Big] \\		
			&+ \me^{-\widetilde{D}_{j}(\theta + \alpha)} \widetilde{R}_{j}( \me^{\widetilde{D}_{j}(\theta)} \widetilde{X}_{j+1}(\theta) + \widetilde{\Delta}_{j}(\theta), \theta) \\
			&+ \widetilde{R}_{j}^{(h1)} + \widetilde{R}_{j}^{(h2)} + \widetilde{R}_{j}^{(h3)},
		\end{split}
	\end{equation}
	where
	\begin{align*}
		\widetilde{R}_{j}^{(h1)}
		& = \left( \sum_{n=2}^{\infty} \frac{(-\widetilde{D}_{j}(\theta+\alpha))^{n}}{n!} \right) (A + V(\theta) + v_{j+1}(\theta)) \me^{\widetilde{D}_{j}(\theta)} \widetilde{X}_{j+1}(\theta) \\
		\nonumber 
		& \quad + (A + V(\theta) + v_{j+1}(\theta)) \left( \sum_{n=2}^{\infty} \frac{(\widetilde{D}_{j}(\theta))^{n}}{n!} \right) \widetilde{X}_{j+1}(\theta) \\
		\nonumber 
		& \quad - \widetilde{D}_{j}(\theta+\alpha) (A + V(\theta) + v_{j+1}(\theta)) \left( \sum_{n=1}^{\infty} \frac{(\widetilde{D}_{j}(\theta))^{n}}{n!} \right) \widetilde{X}_{j+1}(\theta),\\
		\widetilde{R}_{j}^{(h2)}
		& = \left( \sum_{n=1}^{\infty} \frac{(-\widetilde{D}_{j}(\theta+\alpha))^{n}}{n!} \right) \mathbb{W}_{j2}(\theta) \me^{\widetilde{D}_{j}(\theta)}\widetilde{X}_{j+1}(\theta) \\
		\nonumber
		& \quad + \mathbb{W}_{j2}(\theta) \left(\sum_{n=1}^{\infty} \frac{(\widetilde{D}_{j}(\theta))^{n}}{n!} \right) \widetilde{X}_{j+1}(\theta),
	\end{align*}
	and
	\begin{equation*}
		\widetilde{R}_{j}^{(h3)} = \me^{-\widetilde{D}_{j}(\theta + \alpha)} \mathbb{W}_{j2}(\theta) \widetilde{\Delta}_{j}(\theta).
	\end{equation*}
Here, we omit the parameter $ \lambda \in \mathcal{O} $ for convenience because $ \lambda $ is always constant in this lemma.
	The key points are to find $\widetilde{D}_j$ and $\widetilde{\Delta}_j$ such that
\begin{equation}\label{hoeq1}
\begin{split}
(A+V(\theta)&+v_{j+1}(\theta))
\widetilde{D}_j(\theta)
-\widetilde{D}_j(\theta+\alpha) (A+V(\theta)+v_{j+1}(\theta))= -\mathbb{W}_{j2}(\theta),
\end{split}
\end{equation}
and
\begin{equation}\label{hoeq2}
(A+V(\theta)+v_{j+1}(\theta))
\widetilde{\Delta}_j(\theta)
-\widetilde{\Delta}_{j}(\theta+\alpha)=
-\widetilde{U}_j(\theta).
\end{equation}
	
	$ \spadesuit \spadesuit:$ \textbf{Approximately solve $ \widetilde{D}_{j} $ and $ \widetilde{\Delta}_{j} $.}
We first give the discussions with homological equation \eqref{hoeq2} for $ \widetilde{\Delta}_{j}= \big(\widetilde{\delta}_j,\overline{\widetilde{\delta}}_j \big)^{\mathrm{T}}.$ Denote by $\widetilde{U}_j=(\widetilde{u}_j,\overline{\widetilde{u}}_j)^{\mathrm{T}}$ and then due to the diagonal  structure of $A,V,v_{j+1}$, the equation \eqref{hoeq2} is equivalent to
	\begin{equation}\label{hoeq3}
		\big(\me^{\mathrm{i} 2\pi \lambda} + G(\theta) + g_{j+1}(\theta)\big) \widetilde{\delta}_j(\theta) - \widetilde{\delta}_{j}(\theta+\alpha) = -\widetilde{u}_j(\theta),
	\end{equation}
	where $ \|G\|_{r,\mathcal{O}}\leq \varepsilon_{0}^{\frac{1}{3}} $ and $ \|g_{j+1}\|_{\widetilde{r}_{j},\mathcal{O}}\leq \sum_{m=0}^{j}\widetilde{\varepsilon}_{m}^{\frac{1}{2}}\leq 3\widetilde{\varepsilon}_{0}^{\frac{1}{2}}.$

Note that the right hand of equation~\eqref{eqj} corresponds to an area-preserving map, thus
	\begin{equation*}
		\det(A+V+v_j+\widetilde{W}_j) \equiv 1,
	\end{equation*}
	which implies $|\me^{\mathrm{i} 2\pi \lambda} + G|^2=1+O(\widetilde{\varepsilon }_0^{\frac{1}{2}}).$ To solve
\eqref{hoeq3}, we try to rewrite $\me^{\mathrm{i} 2\pi \lambda}+G$ in polar coordinate.
By Lemma~\ref{expoB} in Appendix, there exist real-valued functions $\rho(\theta)$ and $B(\theta)$ with $\|\rho\|_{r,\mathcal{O}}=O(\widetilde{\varepsilon }_0^{\frac{1}{2}}) , \, \|B\|_{r,\mathcal{O}}=O(\|G\|_{r,\mathcal{O}})
=O(\varepsilon_{0}^{\frac{1}{3}}) \ll 1$
	such that
	\begin{equation}\label{expo}
		\me^{\mathrm{i} 2\pi \lambda} + G(\theta)=( 1+\rho(\theta) ) \me^{\mathrm{i} 2\pi (\lambda+B(\theta))}.
	\end{equation}
	It follows that \eqref{hoeq3} can be written as
\begin{equation}\label{hoeq4}
\me^{\mathrm{i} 2\pi (\lambda+B(\theta))} \widetilde{\delta}_j(\theta)+b(\theta)
\widetilde{\delta}_j(\theta)-
\widetilde{\delta}_j(\theta+\alpha)=
-\widetilde{u}_j(\theta),
\end{equation}
where
\begin{equation*}
b(\theta)=\rho(\theta) \me^{\mathrm{i} 2\pi (\lambda+B(\theta))}+g_{j+1}(\theta).
\end{equation*}

Now, we apply Proposition~\ref{solveequation} with $l=1$ to get an approximate solution to \eqref{hoeq4}.
	It only needs to verify that the equation \eqref{hoeq4} satisfies the assumptions \eqref{bnorm} and \eqref{bnorms}. Combining with \eqref{ddefinitionparameters0} in Lemma~\ref{parametersestimates}, one has
	\begin{equation*}
		\|b(\theta,\lambda)\|_{\widetilde{r}_j,\mathcal{O}} \leq \|\rho\|_{\widetilde{r}_j,\mathcal{O}}\| \me^{\mathrm{i} 2 \pi (\lambda+B)}\|_{\widetilde{r}_j,\mathcal{O}} + \|g_{j+1}\|_{\widetilde{r}_j,\mathcal{O}}
		\leq \widetilde{\varepsilon}_{0}^{\frac{1}{3}}\leq Q_{n+1}^{-3\tau^2}\leq \gamma^{2}(12 Q_{n+1}^{2\tau^2})^{-1}
	\end{equation*}
provided $Q_{n+1}^{\tau^2} \geq 12 \gamma^{-2}.$
Moreover, by Lemma~\ref{trunctes} and $Q_{n+1}\leq \overline{Q}_{n}^{\mathbb{A}^4}$, we have
\begin{equation*}
\begin{split}		\|\mathcal{R}_{\overline{Q}_n}B\|_{\frac{r}{2},\mathcal{O}}
& \leq \exp \Big\{ -2^{-1}r^{-1}\Gamma(\pi \overline{Q}_n {r}) \ln (\pi \overline{Q}_n {r}) \Big\} \cdot  \|B\|_{r,\mathcal{O}}\\
&= \exp \Big\{ -2^{-1}\Gamma( \pi \overline{Q}_n\cdot\overline{Q}_{n-1}^{-2} {r_0}) \ln (\pi \overline{Q}_n\cdot\overline{Q}_{n-1}^{-2} {r_0}) \Big\} \cdot \|B\|_{r,\mathcal{O}} \\
&\leq \exp \Big\{ -2^{-1}\Gamma(\overline{Q}_n^{\frac{1}{3}})\ln (\overline{Q}_n^{\frac{1}{3}}) \Big\}\\
& = \overline{Q}_n^{-6^{-1} \Gamma(\overline{Q}_n^{\frac{1}{3}})}
\leq \overline{Q}_n^{-54 \mathbb{A}^8 \tau^4 }\leq
\gamma^{2}({480 \pi^2}Q_{n+1}^{2\tau^2})^{-1}.
\end{split}
\end{equation*}
By the assumption, $ \lambda+B(\theta) $ satisfies $ \lambda + [B(\theta)]_{\theta}\in DC_{\alpha}(\gamma,\tau,K,\mathcal{O}) $.
Therefore, we can apply Proposition~\ref{solveequation} to get an approximate solution $\widetilde{\delta}_j$ to the equation \eqref{hoeq4} with the error term $\widetilde{\delta}_j^{(er)}$. Combining with Lemma~\ref{parametersestimates}, we have
	\begin{equation}\label{so-con1}
		\|\widetilde{\delta}_j\|_{\widetilde{r}_j,\mathcal{O}} \leq  32 \gamma^{-2} Q_{n+1}^{2\tau^2}
		\widetilde{\varepsilon}_j \leq Q_{n+1}^{3\tau^2}
		\widetilde{\varepsilon}_j\leq \widetilde{\varepsilon}_j^{\frac{5}{6}},
	\end{equation}
since $ Q_{n+1}^{\tau^{2}} \geq 32 \gamma^{-2}$ and
\begin{equation}\label{er-con1}
\begin{split}			\|\widetilde{\delta}_j^{(er)}\|_{\widetilde{r}_j
-\widetilde{r}_0\sigma,\mathcal{O}}
& \leq 16 \exp\{ -\widetilde{r}_0\sigma\widetilde{r}_j^{-1} \Gamma(2\pi K(\widetilde{r}_j-\widetilde{r}_0\sigma)) \ln(2\pi K(\widetilde{r}_j-\widetilde{r}_0\sigma)) \} \widetilde{\varepsilon}_j\\
&\leq 16 \exp \{ -\sigma\Gamma(\overline{Q}_{n+1}^{\frac{1}{3}})
\ln(\overline{Q}_{n+1}^{\frac{1}{3}})\} \widetilde{\varepsilon}_j\\
&\leq 16 \me^{-2\mathbb{A}^4\tau^2} \widetilde{\varepsilon}_j
\end{split}
\end{equation}
by $ 2\pi K(\widetilde{r}_j-\widetilde{r}_0\sigma)\geq \overline{Q}_{n+1}^{\frac{1}{3}}$ and the fact $\Gamma(x)\ln(x)$ is non-decreasing on $(1,\infty).$
	
Now, we turn to the equation \eqref{hoeq1} about $ \widetilde{D}_{j} $, which is equivalent to
\begin{equation*}
\begin{split}
(\me^{\mathrm{i} 2\pi \lambda} + G(\theta) + g_{j+1}(\theta))\widetilde{d}_j(\theta) &-\widetilde{d}_j(\theta+\alpha) (\me^{-\mathrm{i} 2\pi \lambda} + \overline{G}(\theta) + \overline{g}_{j+1}(\theta))
=-\widetilde{W}_{j2}(\theta).
\end{split}
	\end{equation*}
By applying \eqref{expo} and direct calculations, we can rewrite the equation above as
\begin{equation}\label{hoeq7}
\me^{\mathrm{i} 4\pi (\lambda+B(\theta))}\widetilde{d}_j(\theta)
+\widetilde{b}(\theta)\widetilde{d}_j(\theta) -\widetilde{d}_j(\theta+\alpha) =w_{j}(\theta),
\end{equation}
where
\begin{equation*}
\begin{split}
\widetilde{b}(\theta)&= \frac{g_{j+1}(\theta) - \me^{\mathrm{i} 4\pi (\lambda+B(\theta))}\overline{g}_{j+1}(\theta)} {(1+\rho(\theta))\me^{-\mathrm{i} 2\pi (\lambda+B(\theta))} + \overline{g}_{j+1}(\theta)},\\
w_{j}(\theta) &= \frac{-\widetilde{W}_{j2}(\theta)}{(1+\rho(\theta)) \me^{\mathrm{-i} 2\pi (\lambda+B(\theta))}+\overline{g}_{j+1}(\theta)}.
\end{split}
\end{equation*}
Then, the estimates $\|\rho\|_{r,\mathcal{O}}=O(\widetilde{\varepsilon }_0^{\frac{1}{2}})$ and $\|B\|_{r,\mathcal{O}}
=O(\varepsilon_{0}^{\frac{1}{3}})$ in Lemma~\ref{expoB} imply
	\begin{equation*}
		\|\widetilde{b}\|_{\widetilde{r}_{j},\mathcal{O}}
		= O(\|g_{j+1}\|_{\widetilde{r}_{j},\mathcal{O}})
		\leq \widetilde{\varepsilon}_0^{\frac{1}{3}},\quad
		\|w_{j}\|_{\widetilde{r}_{j},\mathcal{O}}
		= O(\|\widetilde{W}_{j}\|_{\widetilde{r}_{j},\mathcal{O}})
		\leq c \widetilde{\varepsilon}_j^{\frac{1}{2}},
	\end{equation*}
where $ c \geq 1 $ is a uniform constant independent of $ j $ for small enough $ \varepsilon_{0} $.
	Similar to the discussions about \eqref{hoeq4}, it is obvious that equation \eqref{hoeq7} satisfies the assumptions of Proposition~\ref{solveequation} with $l=2$.
Therefore, equation \eqref{hoeq7} has an approximate solution $\widetilde{d}_j$ with an error term $\widetilde{d}_j^{(er)}$ satisfying
	\begin{equation}\label{so-con2}
		\|\widetilde{d}_j\|_{\widetilde{r}_j,\mathcal{O}}
		\leq 32 c \gamma^{-2} Q_{n+1}^{2\tau^2}
		\widetilde{\varepsilon}_j^{\frac{1}{2}}
		\leq Q_{n+1}^{3\tau^2}
		\widetilde{\varepsilon}_j^{\frac{1}{2}}
		\leq \widetilde{\varepsilon}_j^{\frac{1}{3}}
	\end{equation}
and
	\begin{equation}\label{er-con2} \|\widetilde{d}_j^{(er)}\|_{\widetilde{r}_j-\widetilde{r}_0\sigma,\mathcal{O}}
		\leq 16 c \me^{-2\mathbb{A}^4\tau^2} \widetilde{\varepsilon}_j^{\frac{1}{2}}.
	\end{equation}

As a conclusion, we obtain approximate solutions $ \widetilde{D}_{j} $ and $\widetilde{\Delta}_j$ of the form
	\begin{equation*}
		\widetilde{D}_{j}(\theta) =
		\left( \begin{array}{cc}
			0 & \widetilde{d}_{j}(\theta)\\
			\overline{\widetilde{d}_{j}}(\theta) & 0
		\end{array}
		\right)\in su(1,1),
		\ \
		\widetilde{\Delta}_{j}(\theta) = \left(\widetilde{\delta}_{j}(\theta), \overline{\widetilde{\delta}_{j}}(\theta)\right)^{\mathrm{T}}\in \mathcal{C}^{2},
	\end{equation*}
to the homological equations \eqref{hoeq1} and \eqref{hoeq2} with error terms
	\begin{equation*}
		\widetilde{D}_{j}^{(er)}(\theta):=
		\left( \begin{array}{cc}
			0 & \widetilde{d}_{j}^{(er)}(\theta)\\
			\overline{\widetilde{d}_{j}^{(er)}}(\theta) & 0
		\end{array}
		\right),
		\ \
		\widetilde{\Delta}_{j}^{(er)}(\theta) := \left(\widetilde{\delta}_{j}^{(er)}(\theta), \overline{\widetilde{\delta}_{j}^{(er)}}(\theta)\right)^{\mathrm{T}}.
	\end{equation*}
Therefore, we get the area-preserving transformation
\begin{equation*}
\widetilde{\Phi}_{j}:  \Lambda_{\widetilde{r}_{j+1}}(\mathbb{T}\times \mathcal{O}, \mathcal{C}^{2}) \rightarrow  \Lambda_{\widetilde{r}_{j}}(\mathbb{T}\times \mathcal{O}, \mathcal{C}^{2})
	\end{equation*}
	of the form
	\begin{equation*}
		\widetilde{X}_{j} (\theta, \lambda)
		= \me^{\widetilde{D}_{j}(\theta, \lambda)} \widetilde{X}_{j+1} (\theta, \lambda)
		+ \widetilde{\Delta}_{j} (\theta, \lambda).
	\end{equation*}

	$ \spadesuit\spadesuit\spadesuit: $ \textbf{The new equation.}
	We go back to the new equation~\eqref{j1} about $ \widetilde{X}_{j+1} $.
	Following the choice of $\widetilde{D}_j$ and $\widetilde{\Delta}_j $ by solving the homological equations \eqref{hoeq1} and \eqref{hoeq2}, we know that the equation \eqref{j1} can be simplified as
	\begin{align}\label{202207206}
			\widetilde{X}_{j+1}(\theta + \alpha, \lambda)
			& = (A(\lambda) + V(\theta, \lambda) + v_{j+1}(\theta, \lambda)) \widetilde{X}_{j+1}(\theta, \lambda) \\
			\nonumber
			&\quad + \widetilde{D}_{j}^{(er)}(\theta, \lambda) \widetilde{X}_{j+1}(\theta, \lambda)
			+ \me^{-\widetilde{D}_{j}(\theta + \alpha, \lambda)}  \widetilde{\Delta}_{j}^{(er)}(\theta, \lambda) \\
			\nonumber
			& \quad + \me^{-\widetilde{D}_{j}(\theta + \alpha, \lambda)} \widetilde{R}_{j}(\me^{\widetilde{D}_{j}(\theta, \lambda)} \widetilde{X}_{j+1}(\theta, \lambda) + \widetilde{\Delta}_{j}(\theta, \lambda), \theta, \lambda) \\
			\nonumber
			& \quad	+ \widetilde{R}_{j}^{(h1)} + \widetilde{R}_{j}^{(h2)} + \widetilde{R}_{j}^{(h3)}\\
			\nonumber
			& := (A(\lambda) + V(\theta, \lambda) + v_{j+1}(\theta, \lambda)) \widetilde{X}_{j+1}(\theta, \lambda) + \widetilde{U}_{j+1}(\theta, \lambda)\\
			\nonumber
			& \quad +  \widetilde{W}_{j+1}(\theta,\lambda)\widetilde{X}_{j+1}(\theta,\lambda)+\widetilde{R}_{j+1}(\widetilde{X}_{j+1}(\theta,\lambda), \theta, \lambda),
\end{align}
	where
	\begin{equation*}
		\begin{split}
			\widetilde{U}_{j+1}(\theta,\lambda)
			& =
			\me^{-\widetilde{D}_{j}(\theta + \alpha, \lambda)}
			\Big( \widetilde{\Delta}_{j}^{(er)}(\theta, \lambda) + \mathbb{W}_{j2}(\theta, \lambda) \widetilde{\Delta}_{j}(\theta, \lambda)
			+  \widetilde{R}_{j}(\widetilde{\Delta}_{j}(\theta, \lambda), \theta, \lambda) \Big), \\
			\widetilde{W}_{j+1}(\theta,\lambda)
			& = \widetilde{D}_{j}^{(er)}(\theta, \lambda)
			+ (A(\lambda) + V(\theta, \lambda) + v_{j+1}(\theta, \lambda) ) \left( \sum_{n=2}^{\infty} \frac{(\widetilde{D}_{j}(\theta, \lambda))^{n}}{n!} \right) \\
			& \quad + \left( \sum_{n=2}^{\infty} \frac{(-\widetilde{D}_{j}(\theta+\alpha, \lambda))^{n}}{n!} \right) (A(\lambda) + V(\theta, \lambda) + v_{j+1}(\theta, \lambda)) \me^{\widetilde{D}_{j}(\theta, \lambda)} \\
			& \quad - \widetilde{D}_{j}(\theta+\alpha, \lambda) ( A(\lambda) + V(\theta, \lambda) + v_{j+1}(\theta, \lambda)) \left( \sum_{n=1}^{\infty} \frac{(\widetilde{D}_{j}(\theta, \lambda))^{n}}{n!} \right)\\
			& \quad + \left( \sum_{n=1}^{\infty} \frac{(-\widetilde{D}_{j}(\theta+\alpha, \lambda))^{n}}{n!} \right) \mathbb{W}_{j2}(\theta, \lambda) \me^{\widetilde{D}_{j}(\theta, \lambda)} \\
			& \quad + \mathbb{W}_{j2}(\theta, \lambda) \left(\sum_{n=1}^{\infty} \frac{(\widetilde{D}_{j}(\theta, \lambda))^{n}}{n!} \right) \\
			& \quad + \me^{-\widetilde{D}_{j}(\theta + \alpha, \lambda)} \cdot \partial_{1} \widetilde{R}_{j} (\widetilde{\Delta}_{j}(\theta, \lambda), \theta, \lambda) \cdot \me^{\widetilde{D}_{j}(\theta, \lambda)},
		\end{split}
	\end{equation*}
	and
	\begin{equation*}
		\begin{split}
			& \widetilde{R}_{j+1}(\widetilde{X}_{j+1}(\theta,\lambda), \theta, \lambda)
			= \me^{-\widetilde{D}_{j}(\theta+\alpha, \lambda)}\big\{ \widetilde{R}_{j}( \me^{\widetilde{D}_{j}(\theta, \lambda)} \widetilde{X}_{j+1}(\theta, \lambda) + \widetilde{\Delta}_{j}(\theta, \lambda), \theta, \lambda)  \\
			& \qquad \qquad - \widetilde{R}_{j}(\widetilde{\Delta}_{j}(\theta,\lambda),\theta,\lambda)
			-\partial_{1} \widetilde{R}_{j}(\widetilde{\Delta}_{j}(\theta,\lambda),\theta,\lambda) \cdot \me^{\widetilde{D}_j(\theta,\lambda)} \widetilde{X}_{j+1}(\theta,\lambda)\big\}.
		\end{split}
	\end{equation*}
	Therefore, by the estimates \eqref{so-con1}, \eqref{er-con1}, \eqref{so-con2} and \eqref{er-con2}, we get
	\begin{equation*}
		\begin{split}
			\|\widetilde{U}_{j+1}\|_{\widetilde{r}_{j+1}, \mathcal{O}}
			& \leq \me^{\widetilde{\varepsilon}_{j}^{\frac{1}{3}}} (16 \me^{-2\mathbb{A}^4\tau^2} \widetilde{\varepsilon}_{j} + \widetilde{\varepsilon}_{j}^{\frac{1}{2}} \cdot \widetilde{\varepsilon}_{j}^{\frac{5}{6}} + 4\widetilde{\varepsilon}_{j}^{\frac{5}{3}})
			\leq \me^{-1} \widetilde{\varepsilon}_{j}= \widetilde{\varepsilon}_{j+1},\\
			\|\widetilde{W}_{j+1}\|_{\widetilde{r}_{j+1}, \mathcal{O}}
			& \leq 16 c \me^{-2\mathbb{A}^4\tau^2} \widetilde{\varepsilon}_{j}^{\frac{1}{2}} + 40 \widetilde{\varepsilon}_{j}^{\frac{2}{3}} + 6 \widetilde{\varepsilon}_{j}^{\frac{5}{6}} + 10 \widetilde{\varepsilon}_{j} + 6 \widetilde{\varepsilon}_{j}^{\frac{7}{6}} + 8 \widetilde{\varepsilon}_{j}^{\frac{5}{3}} \\
			& \leq \me^{-\frac{1}{2}} \widetilde{\varepsilon}_{j}^{\frac{1}{2}}= \widetilde{\varepsilon}_{j+1}^{\frac{1}{2}}.
		\end{split}
	\end{equation*}

	Now we turn to the higher order term $\widetilde{R}_{j+1}$. It is obvious that $\widetilde{R}_{j+1}(0,\theta,\lambda)=0$ and $ \partial_{1} \widetilde{R}_{j+1}(0,\theta,\lambda) = 0 $.
	By direct calculation,
	\begin{equation*}
		\begin{split}
			& \quad \,\, \partial_{11}^2\widetilde{R}_{j+1}(\widetilde{X}_{j+1},\theta,\lambda) \\
			& = \me^{-\widetilde{D}_{j}(\theta + \alpha, \lambda)}
			\cdot \partial_{11}^{2} \widetilde{R}_{j} ( \me^{\widetilde{D}_{j}(\theta, \lambda)}
			\widetilde{X}_{j+1}(\theta,\lambda) + \widetilde{\Delta}_j(\theta,\lambda), \theta, \lambda )
			\cdot ( \me^{\widetilde{D}_{j}(\theta, \lambda)} )^{2}.
		\end{split}
	\end{equation*}
	Therefore, we have
	\begin{equation*}
		\begin{split} \|\partial_{11}^2\widetilde{R}_{j+1}\|_{\widetilde{r}_{j+1},\widetilde{s}_{j+1},\mathcal{O}}
			& \leq (1 + 6 \widetilde{\varepsilon}_{j}^{\frac{1}{3}})
			\|\partial_{11}^{2} \widetilde{R}_{j} \|_{\widetilde{r}_j, \widetilde{s}_{j}, \mathcal{O}} \\
			& \leq \Big( \prod_{m=0}^{j} (1 + 6 \widetilde{\varepsilon}_{m}^{\frac{1}{3}}) \Big)
			\|\partial_{11}^{2} \widetilde{R}_{0} \|_{\widetilde{r}_0, \widetilde{s}_{0}, \mathcal{O}},
		\end{split}
	\end{equation*}
	which is indicated by
	\begin{equation*}
		\begin{split}
			& \quad\  \|\me^{\widetilde{D}_{j}(\theta,\lambda)} \widetilde{X}_{j+1}(\theta,\lambda) + \widetilde{\Delta}_{j}(\theta,\lambda)\|_{\widetilde{r}_{j+1},\widetilde{s}_{j+1},\mathcal{O}}  \\
			& \leq  \|\me^{\widetilde{D}_{j}}\|_{\widetilde{r}_{j},\mathcal{O}}\|\widetilde{X}_{j+1}\|_{\widetilde{r}_{j+1},\mathcal{O}}+ \|\widetilde{\Delta}_{j}\|_{\widetilde{r}_{j},\mathcal{O}}\\
			& \leq
			(1 + 2\widetilde{\varepsilon}_{j}^{\frac{1}{3}})\cdot \widetilde{s}_{j+1}+\widetilde{\varepsilon}_j^{\frac{5}{6}}
			\leq
			\widetilde{s}_{j},
		\end{split}
	\end{equation*}
provided $ \varepsilon_{0} \leq (\frac{s_{0}}{240})^3 $.

The discussions above show that the system defined by \eqref{202207206} is the one we want. Up to here, we finish the proof of Lemma~\ref{finite}.
\subsection{Proof of main iterative lemma}\label{sec:fi-step}
In this subsection, we will use an iterative procedure of finite steps to find a transformation $\Phi$, which transforms the equation \eqref{202207182} to a new equation like \eqref{202207182} possesses the same estimates with $n+1$ in place of $n.$
This is achieved by using Lemma~\ref{finite} inductively.

Starting from equation~\eqref{202207182}, we use Lemma~\ref{finite} inductively $ L $ times.
By the choice of $L$ in \eqref{L}, we have $ \widetilde{\varepsilon}_{L}\leq\varepsilon_{n+1}< \widetilde{\varepsilon}_{L-1} $.
Therefore, we terminate the finite iteration at $ L $-th step.
For the definition of $\sigma$ in \eqref{finite-se},
\begin{equation}\label{202207207}
\begin{split}
\widetilde{r}_{L} &= \widetilde{r}_{0} - L\sigma\widetilde{r}_{0}
	= 2^{-1}\widetilde{r}_{0}=r_{n+1},\\
\widetilde{s}_{L} &= \widetilde{s}_{0} - \widetilde{s}_{0}\eta_{n} \sum_{j=0}^{L-1}\eta_{j} \geq \widetilde{s}_{0} - \widetilde{s}_{0}\eta_{n} = s_{n+1}.
\end{split}
\end{equation}
Moreover, we obtain a sequence of area-preserving maps $ \widetilde{\Phi}_{0}, \widetilde{\Phi}_{1}, \cdots, \widetilde{\Phi}_{L-1} $.
	Let
	\[ \Phi_{n}:= \widetilde{\Phi}_{0} \circ \widetilde{\Phi}_{1} \circ \cdots \circ \widetilde{\Phi}_{L-1}, \]
which changes $(Eq)_n$ defined by \eqref{202207182} into $(Eq)_{n+1}:$
	\begin{equation*}
		\begin{split}
			X_{n+1}(\theta+\alpha,\lambda)
			= (&A(\lambda) + V_{n+1}(\theta,\lambda)) X_{n+1}(\theta,\lambda) + U_{n+1}(\theta,\lambda)\\
			& + W_{n+1}(\theta,\lambda) X_{n+1}(\theta,\lambda) + R_{n+1}(X_{n+1}(\theta,\lambda), \theta, \lambda).
		\end{split}
	\end{equation*}
	Namely, $ V_{n+1}= V + v_{L},\, U_{n+1}= \widetilde{U}_{L},\, W_{n+1}= \widetilde{W}_{L},\, R_{n+1}= \widetilde{R}_{L}.$ The two inequalities in \eqref{202207207} show that $({Eq})_{n+1}$ is well-defined on $\Lambda_{r_{n+1}}(\TT\times\mathcal{O},\mathcal{C}^{2}).$
Moreover,
	\begin{equation*}
		\begin{split}
			&\|V_{n+1} - V\|_{r_{n+1},\mathcal{O}}
			\leq \|v_{L}\|_{\widetilde{r}_{L},\mathcal{O}} \leq
			\sum_{m=0}^{L-1} \widetilde{\varepsilon}_m^{\frac{1}{2}} \leq 3\varepsilon^{\frac{1}{2}},
			\\
			&\|U_{n+1}\|_{r_{n+1},\mathcal{O}} \leq \|\widetilde{U}_{L}\|_{\widetilde{r}_{L},\mathcal{O}} \leq
			\widetilde{\varepsilon}_{L} \leq \varepsilon_{n+1},
			\\
			&\|W_{n+1}\|_{r_{n+1},\mathcal{O}} \leq \|\widetilde{W}_{L}\|_{\widetilde{r}_{L},\mathcal{O}} \leq
			\widetilde{\varepsilon}_{L}^{\frac{1}{2}} \leq \varepsilon_{n+1}^{\frac{1}{2}},\\
			& R_{n+1}(0, \theta, \lambda) = \widetilde{R}_{L}(0, \theta, \lambda) = 0,
			\ \,
			\partial_{1} R_{n+1}(0, \theta, \lambda) = \partial_{1} \widetilde{R}_{L}(0, \theta, \lambda) = 0,
		\end{split}
	\end{equation*}
	and
	\begin{equation*}
	\begin{split}
		\|\partial^2_{11}R_{n+1}\|_{r_{n+1},s_{n+1},\mathcal{O}}
		\leq  \|\partial^2_{11}\widetilde{R}_{L}\|_{\widetilde{r}_{L},\widetilde{s}_{L},\mathcal{O}}
		& \leq \Big(\prod_{m=0}^{L-1} (1+6 \widetilde{\varepsilon}_{m}^{\frac{1}{3}})\Big)
		\|\partial^2_{11} R_n \|_{r_n,s_n,\mathcal{O}}\\
	&	\leq \prod_{l=0}^{n} \me^{24 \varepsilon_{l}^{\frac{1}{3}}}.
	\end{split}
	\end{equation*}
Thus, $(Eq)_{n+1}$ defined by above is the one we want in Lemma~\ref{iterationlemma}.
	
	The remaining task is to prove \eqref{202207183}.  It follows from \eqref{trans} in Lemma~\ref{finite} that
	\begin{equation*}
		\begin{split}
			X _n&= \Big[ \prod_{m=0}^{L-1} \me^{\widetilde{D}_{m}} \Big] X_{n+1}
			+ \sum_{m=0}^{L-1} \Big[ \prod_{j=0}^{m} \me^{\widetilde{D}_{j-1}} \Big] \widetilde{\Delta}_{m}\\
			&:= \me^{D_{n+1}} X_{n+1} + \Delta_{n+1},
		\end{split}
	\end{equation*}
	where
	\begin{equation*}
		D_{n+1} = \ln \Big( \prod_{m=0}^{L-1} \me^{\widetilde{D}_m} \Big),
		\qquad \Delta_{n+1} = \sum_{m=0}^{L-1} \Big[ \prod_{j=0}^{m} \me^{\widetilde{D}_{j-1}} \Big] \widetilde{\Delta}_m,
		\qquad \widetilde{D}_{-1} \equiv 0.
	\end{equation*}
Note that $ D_{n+1} $ is a map from $ \mathbb{T} \times \mathcal{O} $ to $su(1,1) $ because $\widetilde{D}_{m}(\theta, \lambda) \in su(1,1),m=0,\cdots,L-1.$
	Moreover, we have
	\begin{equation*}
		\big\| \me^{D_{n+1}} \big\|_{r_{n+1}, \mathcal{O}}
		\leq \prod_{m=0}^{L-1} \me^{\|\widetilde{D}_{m}\|_{\widetilde{r}_{m}, \mathcal{O}}}
		\leq \me^{4 \varepsilon_{n}^{\frac{1}{3}}}
	\end{equation*}
	since $ \|\widetilde{D}_{m}\|_{\widetilde{r}_{m}, \mathcal{O}} \leq \widetilde{\varepsilon}_{m}^{\frac{1}{3}} $.
	Therefore, we obtain
	\begin{equation*}
		\|D_{n+1}\|_{r_{n+1},\mathcal{O}} \leq 4 \varepsilon_{n}^{\frac{1}{3}}.
	\end{equation*}
	Finally, we have
	\begin{equation}\label{conver}
		\|\Delta_{n+1}\|_{r_{n+1},\mathcal{O}}
		\leq \me^{4\varepsilon_{n}^{\frac{1}{3}}} \sum_{m=0}^{L-1}
		\widetilde{\varepsilon}_{m}^{\frac{5}{6}}
		\leq 2 \sum_{m=0}^{L-1} \widetilde{\varepsilon}_{m}^{\frac{5}{6}}
		\leq 4 \varepsilon_{n}^{\frac{5}{6}}.
	\end{equation}
This completes the proof of Lemma~\ref{iterationlemma}, one KAM step.
\subsection{Proof of Theorem~\ref{kamthm}}

We begin with the equation defined by \eqref{kam-eq}:
\begin{align*} \tag{$ {Eq}_{0} $}
	X_0(\theta+\alpha,\lambda)
	& = (A(\lambda)+V_0(\theta,\lambda)) X_0(\theta,\lambda) + U_0(\theta,\lambda) + W_0(\theta,\lambda) X_0(\theta,\lambda) \\
	& \qquad \quad + R_0(X_0(\theta,\lambda), \theta, \lambda)
\end{align*}
for the unknown function $X_0\in \Lambda_{r}(\mathbb{T}\times \mathcal{O}, \mathcal{C}^{2})$ and $ V_{0} \equiv 0.$
Since $ \mathcal{O}=[\frac{1}{4}, \frac{3}{4}] $ is a closed subset in $\mathbb{R}_{+}$, the non-resonance condition $ \lambda \in DC(\gamma,\tau,K_0,\mathcal {O}_0)$ is satisfied by setting
\begin{equation*}
	\begin{split}
		\mathcal{O}_{0}
		&= \mathcal {O} \setminus \bigcup_{|k|< K_{0}} g_{k}^{0}(\gamma _{0}) \\
		&= \Big\{ \lambda\in\mathcal{O}:
		\| l \lambda - k \alpha \|_{\mathbb{T}} \geq \frac{\gamma_{0}}{(|k|+|l|)^\tau},\, \forall\, 0<|k|<K_{0},\, l=1,2\Big\}.
	\end{split}
\end{equation*}
Assume that the parameter $\varepsilon_{0}$ in \eqref{202207218}  satisfies the inequality \eqref{in-s}, then Lemma~\ref{iterationlemma} can be applied to $(Eq)_{0}.$ As a consequences, we get a decreasing sequence of closed sets $ \mathcal{O}_{n-1}$ and a sequence of area-preserving transformations
\begin{equation*}
	\Phi^{n} := \Phi_{0} \circ \cdots \circ \Phi_{n-1}:\,\, \Lambda_{r_n}(\mathbb{T} \times \mathcal{O}_{n-1}, \mathcal{C}^{2}) \rightarrow
	\Lambda_{r}(\mathbb{T} \times \mathcal{O}, \mathcal{C}^{2})
\end{equation*}
such that $\Phi^{n}$ changes $(Eq)_{0}$ to $(Eq)_{n},$ which satisfies the properties in Lemma~\ref{iterationlemma}. By the form \eqref{nform} of each transformation $ \Phi_{n} $, we know that $\Phi^{n}$ have the form
\begin{equation*}
	\begin{split}
		X_{0} &= \Big[ \prod_{j=1}^{n} \me^{D_{j}} \Big] X_{n} + \sum_{j=1}^{n} \Big[ \prod_{m=0}^{j-1} \me^{D_{m}} \Big] \Delta_{j}\\
		&:= \me^{\mathfrak{D}_{n}} X_{n} + \Xi_{n},
	\end{split}
\end{equation*}
where
\begin{equation*}
	\mathfrak{D}_{n} = \ln \Big( \prod_{j=1}^{n} \me^{D_{j}} \Big) \in su(1,1),
	\qquad \Xi_{n} = \sum_{j=1}^{n} \Big[ \prod_{m=0}^{j-1} \me^{D_{m}} \Big] \Delta_{j},
	\qquad D_{0} \equiv 0.
\end{equation*}
Similar to the calculation of \eqref{conver}, we have
\begin{equation*}
	\begin{split}
		\| \mathfrak{D}_n \|_{r_n,\mathcal{O}_{n-1}}
		& \leq \sum_{j=1}^{n} \|D_j\|_{r_j,\mathcal{O}_{j-1}}
		\leq 4 \sum_{j=1}^{\infty}\varepsilon_{j-1}^{\frac{1}{3}} \leq 8 \varepsilon_0^{\frac{1}{3}},\\
		\|\Xi_{n}\|_{r_n,\mathcal{O}_{n-1}}
		& \leq \sum_{j=1}^{n}\|\Delta_j\|_{r_j,\mathcal{O}_{j-1}}\leq 4 \sum_{j=1}^{\infty}\varepsilon_{j-1}^{\frac{5}{6}}\leq 8 \varepsilon_0^{\frac{5}{6}}.
	\end{split}
\end{equation*}
Therefore, we know that $\Phi^{n}$ converges uniformly under the topologies derived by $\|\cdot\|_{0,\mathcal{O}_*}$ and $\|\cdot\|_{0,s_{0}/2,\mathcal{O}_*}$ to $\Phi_{*}$, where $ \mathcal {O}_{*}=\bigcap_{n=0}^{\infty}\mathcal {O}_{n}.$ Moreover,
note $r_{*}=0,$ the regularity
of $\Phi_{*}$ in $\theta\in\TT$ may not be analytic any more.
The proof that $\Phi_{*}$ is $C^{\infty}$ differentiable in $\theta\in\TT$ is similar to the ones in \cite{KrikorianW18} (Page 18) and \cite{Cheng22} (Page 26), we omit the details.

Let the limits be $ V_{*},\, U_{*}(=0),\, W_{*}(=0),\, R_{*},\, \Phi_{*} $.
It follows from  \eqref{bestimate} that
\begin{equation*}
	\begin{split}
		&\|V_{*}\|_{0,\mathcal {O}_{*}}
		\leq \sum_{n=1}^{\infty}\|V_{n}-V_{n-1}\|_{r_{n},\mathcal {O}_{n-1}}
		\leq 4\varepsilon _{0}^{\frac{1}{3}},\\
		&R_*(0,\theta,\lambda)=0,
		\quad \partial_{1} R_*(0,\theta,\lambda)=0,
		\quad \|\partial^2_{11} R_*\|_{0,s_{0}/2,\mathcal{O}_{*}} \leq \prod_{j=0}^{\infty} \me^{24 \varepsilon_{j}^{\frac{1}{3}}} \leq 2.
	\end{split}
\end{equation*}
As a conclusion, for any $ \lambda \in \mathcal{O}_{\ast} $, we obtain the finally invariance equation
\begin{equation*}
	X_{\ast} (\theta+\alpha,\lambda) = \big(A(\lambda) + V_*(\theta,\lambda)\big) X_*(\theta,\lambda)
	+ R_*(X_*(\theta,\lambda), \theta, \lambda).
\end{equation*}

\subsection{Measure estimates}\label{sec:mea}
In the process of KAM iteration, we obtain a decreasing sequence of closed sets $\mathcal {O}_{0}\supset \mathcal {O}_{1}\supset \cdots$. It is crucial to prove that the Lebesgue measure of their intersection $ \mathcal{O}_{\ast} = \bigcap_{n=0}^{\infty}\mathcal {O}_{n}$ is nonzero. For more information of the measure estimate, we refer to \cite{poschel89,XJunX97}.
\begin{lemma}\label{measurelemma}
	For the subset $\mathcal{O}_{*}=\bigcap_{n=0}^{\infty}\mathcal {O}_{n}$ of $\mathcal{O}$, where $\mathcal{O}_{n}$ is defined in \eqref{ta-out}, we have
	$ \meas (\mathcal{O}\setminus\mathcal {O}_{*})= O(\gamma)$ for sufficiently small $\gamma>0$.
\end{lemma}
\begin{proof}
	The complete proof
	can be found in Section $4.3$ in \cite{WangY17}. To make the paper more self-contained, we give the sketch of the ideas.
	By Lemma~\ref{expoB}, we have
	\begin{equation*}
		\begin{split}
			\Big| \frac{\dif}{\dif \lambda} \big(l(\lambda+[B_{n}(\theta,\lambda)]_{\theta}) - k\alpha\big) \Big|
			& \geq |l| - \left\| [B_{n}(\theta,\xi)]_{\theta} \right\|_{\mathcal{O}_{n-1}}\\
			& \geq |l| - c \left\| V_{n}(\theta,\xi)_{\theta} \right\|_{\mathcal{O}_{n-1}}\\
			& \geq |l| - 4c \varepsilon_{0}^{\frac{1}{3}}
			\geq 2^{-1}
		\end{split}
	\end{equation*}
	for any $ n\geq 0 $, $ l = 1, 2 $ and $ \lambda\in \mathcal{O}_{n-1} $.
	Combining with \eqref{out}, one gets
	\begin{equation*}
		\meas \big(g_{k}^{n}(\gamma _{n})\big)\leq 2\gamma_{n}|k|^{-\tau}.
	\end{equation*}
	Therefore,
	\begin{equation*}
		\begin{split}
			& \meas \big( \bigcup_{0<|k|< K_{0}} g_{k}^{0}(\gamma_{0}) \big) \leq 2 \sum_{0<|k|< K_{0}} \frac{\gamma_{0}}{|k|^{\tau}},\\
			& \meas \big( \bigcup_{K_{n-3} \leq |k|< K_{n}} g_{k}^{n}(\gamma_{n}) \big) \leq
			2 \sum_{K_{n-3} \leq |k|< K_{n}} \frac{\gamma_{n}}{|k|^{\tau}},\quad \forall\, n\geq1.
		\end{split}
	\end{equation*}
	Moreover, note $\gamma_{n} = \gamma_0\eta_{n}$, then we have the following
	\begin{equation*}
		\begin{split}
			& \quad \ 2 \sum_{0<|k|< K_{0}} \frac{\gamma_{0}}{|k|^{\tau}} +
			2 \sum_{n=1}^{\infty} \sum_{K_{n-3} \leq |k|< K_{n}} \frac{\gamma_{n}}{|k|^{\tau}} \\
			& \leq 2 \gamma_{0} \sum_{n=1}^{\infty} \frac{1}{(n+2)^2} \sum_{\kappa=1}^{\infty} \sum_{|k|=\kappa} \frac{1}{|k|^{\tau}} \\
			& \leq 4 \gamma_{0} \sum_{n=1}^{\infty} \frac{1}{(n+2)^2} \sum_{\kappa=1}^{\infty} \frac{1}{\kappa^{\tau}}
			= O(\gamma_{0})
		\end{split}
	\end{equation*}
	by $\tau>1$.
	That is, $ \meas (\mathcal {O} \setminus \mathcal {O}_{*}) = O(\gamma) $.
\end{proof}

\section*{Appendix}\label{sec:appendix}
\begin{lemma}\label{expoB}
	If $\|G\|_{r,\mathcal{O}}=O(\varepsilon_{0})$ and $|\me^{\mathrm{i 2\pi \lambda}}+G|=1+O(\varepsilon_{n})$, then there exist real-valued functions $\rho$ and $B$ satisfying
	\begin{equation}\label{eq}
		\me^{\mathrm{i 2\pi \lambda}} + G(\theta,\lambda) = (1+\rho(\theta,\lambda)) \me^{\mathrm{i} 2 \pi(\lambda+B(\theta,\lambda))}
	\end{equation}	
	with
	\begin{equation*}
		\|\rho\|_{r,\mathcal{O}}=O(\varepsilon_{n}),
		\qquad \|B\|_{r,\mathcal{O}}=O(\varepsilon_{0}).
	\end{equation*}	
\end{lemma}	
\begin{proof}
	The result can be obtained directly from the following graph.
	\begin{figure}[htbp]
		\includegraphics[width=0.5\textwidth]{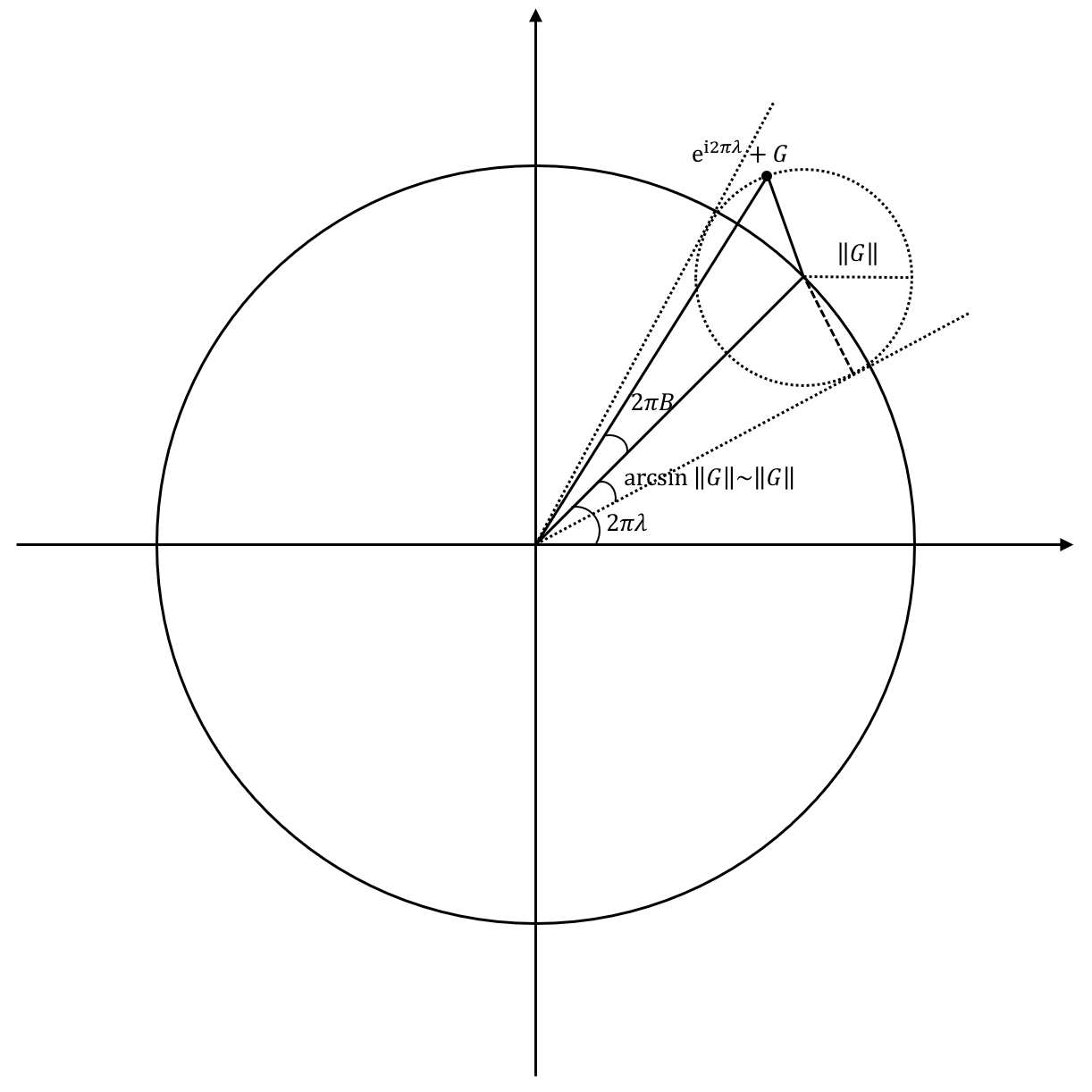}
		\caption{The increment of the argument $ 2\pi B $.}
	\end{figure}
	In detail, the function $\rho$ is obvious since the modulus of $\me^{\mathrm{i 2\pi \lambda}} + G$ is $1+O(\varepsilon_{n})$. We now find the function $B$.
	Let $G=G_1+\mathrm{i} G_2$ with $G_i \in \Lambda_{r}( \mathbb{T} \times \mathcal{O}, \mathbb{R})$ and $\|G_i\|_{r,\mathcal{O}} = O(\|G\|_{r,\mathcal{O}}),\,(i=1,2)$. Then, \eqref{eq} is equivalent to
	\begin{equation*}
		\left\{
		\begin{array}{c}
			(1+\rho)\cos 2\pi(\lambda+B) = \cos 2\pi\lambda+G_1,\\
			(1+\rho)\sin 2\pi(\lambda+B) = \sin 2\pi\lambda+G_2.
		\end{array}
		\right.
	\end{equation*}
	By direct calculation, we have
	\begin{equation*}
		\tan 2\pi B=\frac{G_2 \cos 2\pi\lambda - G_1 \sin 2\pi \lambda}{1+G_1\cos 2\pi\lambda+G_2\sin 2\pi\lambda}.
	\end{equation*}
	Hence $ \|B\|_{r,\mathcal{O}} = O(\|G\|_{r,\mathcal{O}})=O(\varepsilon_{0}) $.

\end{proof}

At the end, we present the proof of Lemma~\ref{bettersmalldivisor} for small divisor condition.
\begin{proof}[Proof of the Lemma~\ref{bettersmalldivisor}]
	We prove the lemma in two cases.
	
	\textbf{Case 1.} $\overline{Q}_{n+1}\leq  Q_{n+1}^{2\tau}$. It follows that $|k|\leq K = \overline{Q}_{n+1}^{\frac{1}{2}}\leq Q_{n+1}^{\tau}$. Therefore, one has
	\begin{equation*}
		\big| \me^{\mathrm{i} 2\pi (l \Omega(\lambda) - k\alpha)}-1\big| = 2\big| \sin \pi (l \Omega(\lambda) - k\alpha) \big|
		\geq 4 \| l \Omega(\lambda) - k \alpha \|_{\mathbb{T}}
		\geq \frac{4\gamma}{|k|^{\tau}}
		\geq \frac{4\gamma}{Q_{n+1}^{\tau^2}}
	\end{equation*}
	for $ \Omega(\lambda) \in DC_{\alpha}(\gamma, \tau, K, \mathcal{O}) $ and $l=1,2.$
	
	\textbf{Case 2.} $\overline{Q}_{n+1}> Q_{n+1}^{2\tau}$. 	 In this case, for any $|k|\leq K=\overline{Q}_{n+1}^{\frac{1}{2}}$, we decompose it as
	\begin{equation*}
		k=\widetilde{k}+mQ_{n+1},
	\end{equation*}
	where $|\widetilde{k}|<Q_{n+1}$. Then,
	\begin{equation*}
		|mQ_{n+1}|\leq |k|+|\widetilde{k}|\leq \overline{Q}_{n+1}^{\frac{1}{2}}+Q_{n+1},
	\end{equation*}
	which yields
	\begin{equation*}
		|m|<\overline{Q}^{\frac{1}{2}}_{n+1}Q_{n+1}^{-1}+1.
	\end{equation*}
	It follows that
	\begin{align*}
		\big| \me^{\mathrm{i} 2\pi (l \Omega(\lambda) - k\alpha)} - 1 \big|
		& \geq 4 \|l \Omega(\lambda) - k\alpha\|_{\mathbb{T}}\\
		& = 4\|l \Omega(\lambda) - \widetilde{k}\alpha-mQ_{n+1}\alpha\|_{\mathbb{T}}\\
		& \geq 4 \|l \Omega(\lambda) - \widetilde{k}\alpha\|_{\mathbb{T}} - 4 |m|\|Q_{n+1}\alpha\|_{\mathbb{T}}\\
		& \geq4 \gamma Q_{n+1}^{-\tau} - 4 (\overline{Q}^{\frac{1}{2}}_{n+1}Q_{n+1}^{-1}+1)
\overline{Q}_{n+1}^{-1}\\
		& \geq \frac{4 \gamma}{Q_{n+1}^{\tau}}-\frac{8}{Q_{n+1}^{\tau+1}}
		 \geq \frac{2 \gamma}{Q_{n+1}^{\tau}}
		\geq \frac{4 \gamma}{Q_{n+1}^{\tau^2}}
	\end{align*}
	since we choose $ Q_{n+1} $ big enough such that $ Q_{n+1}\geq \frac{4}{\gamma} $.
\end{proof}


\section*{Acknowledgments}
 H. Cheng was
supported by National Natural Science
Foundation of China 
(No.12001294), S. Wang was supported by Shandong Provincial Natural Science Foundation of China (Grant Nos. ZR2021QA022). F.Wang was supported by the National Natural Science
Foundation of China (No.12101434) and the Fundamental Research Funds of Sichuan Normal University (KY. 20200921).



\bibliographystyle{plain}
\bibliography{20220718_ultra_map_liouvillean-reference}

\end{document}